\documentclass{amsart}

\usepackage{amssymb, amsthm}
\usepackage[all]{xy}

\usepackage{ast2}

\usepackage{tikz-cd}

\usepackage{xcolor}

\newcommand{\N}{\mathbb{N}}

\newcommand{\km}{\mathbf k}
\newcommand{\sm}{\mathbf s}
\newcommand{\am}{\mathbf a}
\newcommand{\idm}{\mathbf i}

\newcommand{\Km}{\mathbf K}
\newcommand{\Sm}{\mathbf S}
\newcommand{\Em}{\mathbf E}
\newcommand{\Idm}{\mathbf I}

\newcommand{\cleq}{\preccurlyeq}
\newcommand{\cgeq}{\succcurlyeq}
\newcommand{\bcw}{\bigcurlywedge}

\newcommand{\fst}{\mathsf{fst}}
\newcommand{\snd}{\mathsf{snd}}
\newcommand{\eval}{\mathsf{eval}}

\newcommand{\dom}{\text{dom}}

\title{Arrow algebras}
\author{Benno van den Berg and Marcus Bri\"et}
\date{\today}

\begin{document}

\begin{abstract} 
    In this paper we introduce arrow algebras,  simple algebraic structures which induce elementary toposes through the tripos-to-topos construction. This includes localic toposes as well as various realizability toposes, in particular, those realizability toposes which are obtained from partial combinatory algebras. Since there are many examples of arrow algebras and arrow algebras have a number of closure properties, including a notion of subalgebra given by a nucleus, arrow algebras provide a flexible tool for constructing toposes; we illustrate this by providing some general tools for creating toposes for Kreisel's modified realizability. 
\end{abstract}

\maketitle

\section{Introduction}

\subsection{Contributions}
The aim of this paper is to introduce \emph{arrow algebras}, simple algebraic structures generalising complete Heyting algebras (locales), which give rise to toposes. Indeed, we will show that any arrow algebra induces a tripos, an ``arrow tripos'', from which a topos can be obtained through the tripos-to-topos construction.

Besides localic toposes, the toposes which we obtain in this way include various realizability toposes. Typically, realizability toposes are obtained from partial combinatory algebras (pcas; see \cite{vanoosten08} or \refdefi{pca} below); we will see that every pca gives rise to an arrow algebra, and that the topos associated to the induced arrow algebra is isomorphic to the realizability topos over the pca.

For that reason arrow algebras are similar to Alexandre Miquel's implicative algebras \cite{miquel20}. Indeed, Miquel's work was a major source of inspiration for this paper and the definition of an arrow algebra is obtained by weakening one of the axioms for an implicative algebra; hence implicative algebras are also arrow algebras, but not conversely. 

It is the case, however, that every arrow algebra is equivalent to an implicative algebra in a suitable sense (see \refrema{equivalentimplicativealgebra} below). In particular, Miquel has shown that any set-based tripos is equivalent to an ``implicative tripos'' \cite{miquel20b}.

Still, we felt that it would be worthwile to generalise Miquel's theory. The reason is that there are many naturally occurring examples of arrow algebras which are not implicative algebras, of which we will see many examples in this paper. (For instance, the arrow algebra obtained from a pca will be an implicative algebra if and only if the pca is total.) Moreover, there is a simple construction of subalgebras through nuclei (see \refprop{subarrowalgebras} below) which in general only seems to yield arrow algebras. Also, we think it is better to understand an arrow tripos directly in terms of its associated arrow algebra, rather than through the implicative algebra equivalent to the arrow algebra. Indeed, the construction of the equivalent implicative algebra is too involved to make calculations in that structure feasible (again, see \refrema{equivalentimplicativealgebra} for the construction).

Indeed, we believe the advantage of the notion of an arrow algebra is that it is a simple algebraic structure with various closure properties and with many naturally occurring examples. In this way we expect the theory of arrow algebras to provide us with a more flexible framework for constructing new toposes. We will illustrate this latter point by constructing a number of toposes for Kreisel's modified realizability. 

The main contributions of this paper are therefore as follows:
\begin{enumerate}
    \item We introduce the notion of an arrow algebra and give numerous examples.
    \item We show how for each arrow algebra one can construct an associated \emph{arrow topos}. This construction goes via the tripos-to-topos construction \cite{hylandjohnstonepitts80,vanoosten08}, so, in fact, what we do is construct an associated \emph{arrow tripos}.
    \item We show that arrow algebras are closed under a notion of a subalgebra given by \emph{nuclei}. This is inspired by locale theory where there is a correspondence between sublocales of a locale $A$ and nuclei on a locale $A$. We define a notion of a nucleus on arrow algebra and we show how such a nucleus gives rise to another arrow algebra (which should be thought of as a ``subalgebra''). 
    \item We illustrate the flexibility of the theory of arrow algebras by showing how they can be used to construct various toposes for Kreisel's modified realizability.
\end{enumerate}   

\subsection{Related work} In the literature one can find a number of other attempts to isolate a basic structure which is sufficient for constructing a tripos. Among those we have focused mostly on Miquel's implicative algebras \cite{miquel20}, because our notion of an arrow algebra is a modification of his. Evidenced frames, as introduced by Cohen, Miquey and Tate in \cite{cohenetal21}, are closely related to implicative algebras. 

Another line of research focuses on \emph{basic combinatory objects} (BCOs), as introduced in Hofstra's paper \cite{hofstra06} (see also \cite{vanoostenzou16}). This paper has been very influential in advocating for a ``relative'' setting for studying realizability: indeed, we follow this here by making a separator part of the definition of an arrow algebra. Nevertheless, we have not written this paper in the context of BCOs, because we do not see how to give arrow algebras the structure of BCO in such a way that the induced triposes are isomorphic.

\subsection{Contents.}
The precise contents of our paper are as follows. In Section 2 we recall Miquel's notion of an implicative algebra as well as the notion of a tripos and a partial combinatory algebra. In Section 3 we introduce the notion of an arrow algebra, give a range of examples of arrow algebras and show that they are closed under a notion of subalgebra induced by a nucleus. In Section 4 we show how one can interpret $\lambda$-terms from the untyped $\lambda$-calculus in suitable arrow algebras. This will be used in Section 5, where we show how one can construct triposes from arrow algebras. In Section 6 we will show how the tripos obtained from a subalgebra of an arrow algebra is a subtripos of the tripos obtained from the original algebra. In Section 7 we will use the theory we have developed to construct various toposes for modified realizability.

\subsection{Acknowledgements.} The contents of this paper are based on the MSc thesis of the second author supervised by the first author \cite{briet23}. We thank Umberto Tarantino and the referees for their careful reading of the manuscript and their useful comments.

\section{Implicative algebras}

As said, our arrow algebras originated as a modification of Miquel's notion of an implicative algebra, as introduced in \cite{miquel20}. For that reason we will start by recapitulating Miquel's theory. While doing so, we will also recall the notions of a tripos and a partial combinatory algebra which will become important in the later sections. The main references for this section are \cite{miquel20,vanoosten08,zoethout22}.

\subsection{Definition} In order to define implicative algebras, we first need to define implicative structures.

\begin{defi}{implstr} An \emph{implicative structure} is a complete poset $(A, \cleq)$ together with a binary operation $\to: A \times A \to A$ satisfying the following two conditions:
        \begin{enumerate}
            \item[(1)] If $a' \cleq a$ and $b \cleq b'$ then $a \to b \cleq a' \to b'$.
            \item[(2)] For all $a \in A$ and subsets $B \subseteq A$ we have \[ a \to \bigcurlywedge_{b \in B} b = \bigcurlywedge_{b \in B} (a \to b). \]   
        \end{enumerate}
\end{defi}

\begin{rema}{intuitionforimplstr}
    We think of the elements of $A$ as \emph{truth values} or \emph{pieces of evidence}. However, we should \emph{not} think of $\cleq$ as giving us the logical ordering of these truth values: the logical ordering will be a derived concept. In many examples the ordering $\cleq$ is closer to a notion of \emph{subtyping} or an \emph{information ordering}: one could say that $a \cleq b$ means that any evidence contained in $a$ is also contained in $b$, making $b$ less informative than $a$. We will occasionally refer to this ordering as the \emph{evidential ordering}.
\end{rema}

The next step is that we single out those bits of evidence that we regard as conclusive: or those ``designated'' truth values that we consider to be true. This we do by choosing a \emph{separator}.

\begin{defi}{separator}
Let $A = (A, \cleq, \to)$ be an implicative structure. A \emph{separator} on $A$ is a subset $S \subseteq A$ such that the following are satisfied:
            \begin{enumerate}
                \item[(1)] If $a \in S$ and $a \cleq b$, then $b \in S$.
                \item[(2)] If $a \to b \in S$ and $a \in S$, then $b \in S$.
                \item[(3)] Both $\km$ and $\sm$ belong to $S$. 
            \end{enumerate}  
Here ${\bf k}$ and ${\bf s}$ are defined as follows:
    \begin{eqnarray*}
        {\bf k} & := & \bigcurlywedge_{a,b \in A} a \to b \to a \\
        {\bf s} & := & \bigcurlywedge_{a,b,c \in A} (a \to b \to c) \to (a \to b) \to (a \to c)
    \end{eqnarray*}
\end{defi}

\begin{defi}{implalgebra}
An \emph{implicative algebra} is a quadruple $(A, \cleq, \to, S)$ consisting of an implicative structure $(A, \cleq, \to)$ and a separator $S$ on that implicative structure.
\end{defi}

\begin{rema}{howtoread} We are using the convention that the operation $\to$ associates to the right, and hence $a \to b \to c$ has to be read as $a \to (b \to c)$. In addition, $\to$ binds stronger than $\bigcurlywedge$ and hence the scope of the occurrences of $\bigcurlywedge$ in the definitions of $\bf k$ and $\bf s$ above is both times the entire formula.
\end{rema}    

\subsection{Examples} Although the main emphasis in this paper will be on the more general arrow algebras that will be defined below, it is still good to see a number of examples of implicative algebras.

\subsubsection{Locales} Any locale (or complete Heyting algebra) $L$ is an example of an implicative structure. If we choose the ``canonical'' separator $S = \{ \top \}$, consisting only of the top element of the locale, then we obtain an implicative algebra. However, on a locale separators coincide with filters, so any filter on a locale equips the locale with the structure of an implicative algebra.

\subsubsection{Total combinary algebras} For the next example, we need to give a summary of the theory of partial and total combinatory algebras. This will provide us with some genuine examples of arrow algebras as well. For this summary we will follow the definitions and conventions of Jetze Zoethout's PhD thesis \cite{zoethout22}, which builds on earlier work such as \cite{hofstravanoosten03,hofstra06,stekelenburg13}. The reader should be warned that these are not standard; however, they are quite convenient for our purposes and that is why we shall follow them here.

\begin{defi}{pap} A \emph{partial applicative poset} is a triple $(P, \cdot, \leq)$ where $(P, \leq)$ is a poset and $\cdot$ is a partial binary operation $\cdot \, : A \times A \rightharpoonup A$ called \emph{application} which has downwards closed domain and is monotone on its domain. In other words:
    \begin{quote}
        if $a' \cdot b'$ is defined and $a \leq a'$ and $b \leq b'$, then $a \cdot b$ is defined and $a \cdot b \leq a' \cdot b'$.
    \end{quote}
We say that $P$ is \emph{total} is the application operation is total, and $P$ is \emph{discrete} if the order $\leq$ is a discrete order.
\end{defi}

\begin{rema}{onapplication} Most of the time the symbol for the application operation will be omitted, so that the result of applying $a$ to $b$ will be denoted $ab$ instead of $a \cdot b$. To indicate that $ab$ is defined, we will write $ab \downarrow$ (and $ab \uparrow$ if it is not defined). Furthermore, the convention is that application associates to the left, so $abc$ has to be read as $(ab)c$.
\end{rema}

\begin{defi}{filter}
    Let $(P, \cdot, \leq)$ be a partial applicative poset. A \emph{filter} $P^\#$ on $P$ is a subset $P^\# \subseteq P$ such that:
    \begin{enumerate}
        \item[(i)] if $a,b \in P^\#$ and $ab$ is defined, then $ab \in P^\#$.
        \item[(ii)] if $a \leq b$ and $a \in P^\#$, then $b \in P^\#$.
    \end{enumerate} 
    A quadruple $(P, \cdot, \leq, P^\#)$ consisting of a partial applicative poset $(P, \cdot, \leq)$ together with a filter $P^\#$ on it will be called a \emph{partial applicative structure}. A partial applicative structure will be called \emph{absolute} if $P^\# = P$.
\end{defi}    

\begin{defi}{pca} A partial applicative structure $(P, \cdot, \leq, P^\#)$ will be called a \emph{partial combinatory algebra} (or \emph{pca} for short) if there are elements $\km, \sm \in P^\#$ satisfying:
\begin{enumerate}
    \item[(1)] $\km a b$ is defined and $\km a b \leq a$;
    \item[(2)] $\sm a b \downarrow$;
    \item[(3)] if $ac(bc)$ is defined, then so is $\sm a b c$ and $\sm a b c \leq ac(bc)$,
\end{enumerate}
for all $a,b, c \in P$. A partial combinatory algebra will be called \emph{total} or \emph{discrete} if the underlying partial applicative poset is; and \emph{absolute} if it is absolute as a partial applicative structure.
\end{defi}

A fundamental result in the theory of pcas is that they are \emph{combinatory complete}. To state this property, we need the following definition.

\begin{defi}{termsoverapca} 
    The set of \emph{terms} over a pca $P$ is the smallest collection of formal expressions satisfying the following properties:
    \begin{enumerate}
        \item[(i)] any element of $P$ is a term over $P$.
        \item[(ii)] any variable is a term over $P$.
        \item[(iii)] if $s,t$ are terms over $P$, then so is $st$.
    \end{enumerate} 
\end{defi}

\begin{theo}{combcompleteness} {\rm (Combinatory completeness)}
    Let $P$ be a pca. There exists an operation which given a non-empty sequence of distinct variables $\overline{x} = x_1,\ldots,x_n$, a variable $y$ and a term $t(\overline{x}, y)$ over $P$ with all free variables among $\overline{x}, y$, produces an element $\lambda^* \overline{x},y.t \in P$, which satisfies:
    \begin{enumerate}
        \item[(i)] for any $\overline{a} \in P^n$, the expression $(\lambda^* \overline{x},y.t)\overline{a}$ is defined.
        \item[(ii)] for any $\overline{a} \in P^n$ and $b \in P$, if the expression $t(\overline{a}, b)$ is defined, then so is $(\lambda^* \overline{x},y.t)\overline{a}b$ and
        \[  (\lambda^* \overline{x},y.t)\overline{a}b \leq t(\overline{a}, b). \]
        \item[(iii)] if all constants of $t$ are elements of $P^\#$, then also $\lambda^* \overline{x},y.t \in P^\#$.
    \end{enumerate} 
\end{theo}
\begin{proof}
    In Zoethout's PhD thesis \cite{zoethout22}, this is Proposition 2.1.24.
\end{proof}

Next, let us briefly give some examples of pcas. For more examples (and theory), the reader is referred to \cite{zoethout22,vanoosten08}.
\begin{exam}{examplepca}
    \begin{enumerate}
        \item Let $\mathbb{N}$ be the set of natural numbers and $\varphi_e$ be the $e$th partial recursive function under some coding of the partial recursive functions. Then we can define an application operation, \emph{Kleene application}, on $\mathbb{N}$, by defining $m \cdot n = \varphi_m(n)$, the result of the $m$th partial recursive function on input $n$, whenever defined. This gives us an example of an absolute and discrete partial combinatory algebra, which is denoted by $K_1$, Kleene's first algebra.
        \item On the set $\mathbb{N} \to \mathbb{N}$ of functions from the natural numbers to the natural numbers we can also define a partial application function. To explain how, we need to establish some notation. 
        
        Let us fix some recursive coding of finite sequences with $[a_0,a_1,\ldots,a_{n-1}]$ in $\N$ being a code for the finite sequence $<a_0,a_1,\ldots,a_{n-1}>$. If $\alpha \in \mathbb{N}^\mathbb{N}$ and $n \in \N$, then we will write $\alpha |_n$ for $[\alpha(0),\ldots,\alpha(n-1)]$ and $[n] * \alpha$ for the function $\alpha'$ defined by $\alpha'(0) = n$ and $\alpha'(i+1) = \alpha(i)$. For $\alpha \in \N^\N$ let us define a partial function $F_\alpha: \N^\N \rightharpoonup \N$ by setting $F_\alpha(\beta) = m$ if there exists an $n \in \N$ such that:
        \begin{enumerate}
            \item[(i)] for all $i \lt n$, we have $\alpha(\beta|_i) = 0$;
            \item[(ii)] $\alpha(\beta|_n) = m + 1$.
        \end{enumerate}  
        Then the application $\alpha\beta$ for $\alpha, \beta \in \N^\N$ in Kleene's second model $K_2$ is defined as follows: we have $\alpha\beta \downarrow$ if $F_\alpha([n] * \beta)$ is defined for all $n \in \mathbb{N}$, in which case $(\alpha\beta)(n) = F_\alpha([n] * \beta)$. In this way we obtain an absolute and discrete pca. We can make it ``relative'' by choosing the set of total computable functions as a filter.
        \item As an example of a total combinatory algebra, we can consider the set of $\lambda$-terms modulo $\beta$-equivalence. This gives one an example of an absolute and discrete total combinatory algebra. We can also consider the terms of the lambda calculus with $M \leq N$ if $M \twoheadrightarrow_\beta N$, that is, if $M$ reduces to $N$ in a finite number of $\beta$-reduction steps. In that case we obtain an example of an absolute total combinatory algebra.
    \end{enumerate}
\end{exam}

\begin{theo}{impalgfromtcas}
    Let $(P, \leq, \cdot, P^\#)$ be a \emph{total} combinatory algebra. Then the collection $DP$ of downwards closed subsets of $P$ ordered by inclusion can be equipped with an implicative structure as follows:
    \[ X \to Y := \{ z \in P \, : \, (\forall x \in X) \, zx \in Y \}. \]
    The downsets $X \in DP$ containing an element from the filter form a separator on this implicative structure and hence we obtain an implicative algebra.
\end{theo}
\begin{proof}
    The proof is not difficult: the reader can have a look at the proof of \reftheo{arrowalgfrompcas} below for the idea.
\end{proof}

Unfortunately, this result does not apply to partial combinatory algebras that are not total and this will be one reason for wanting to generalise the theory of implicative algebras.

\subsection{Implicative triposes} A crucial result in the theory of implicative algebras is that each implicative algebra gives rise to a tripos (the ``implicative tripos''), from which a topos can be built using the tripos-to-topos construction (see \cite{hylandjohnstonepitts80,vanoosten08}). Let us start by defining triposes.

\begin{defi}{preHeytingalg} A \emph{preHeyting algebra} is a preorder whose poset reflection is a Heyting algebra. Put differently, it is a preorder which has finite products and coproducts and is cartesian closed as a category. We will write ${\bf preHeyt}$ for the 2-category of preHeyting algebras: its objects are preHeyting algebras, its 1-cells are functors preserving products, coproducts and exponentials and its 2-cells are natural transformations between those. Note that any two parallel 2-cells are identical, and ${\bf preHeyt}$ is preorder-enriched.
\end{defi}

\begin{defi}{tripos}
    A \emph{${\bf Set}$-based tripos} is a pseudofunctor $P: {\bf Set}^{\rm op} \to {\bf preHeyt}$ satisfying the following requirements:
    \begin{enumerate}
        \item For any function $f: J \to I$ the map $Pf: PI \to PJ$ has both a left adjoint $\exists_f$ and a right adjoint $\forall_f$ in the category of preorders. That is, for any $\alpha \in PJ$ we have elements $\exists_f(\alpha), \forall_f(\alpha) \in PI$ such that for any $\beta \in PI$ we have:
        \begin{eqnarray*}
            \exists_f (\alpha) \leq \beta & \Longleftrightarrow & \alpha \leq Pf (\beta) \\
            \beta \leq \forall_f (\alpha) & \Longleftrightarrow & Pf (\beta) \leq \alpha
        \end{eqnarray*}
        \item Moreover, these adjoints satisfy the \emph{Beck-Chevalley condition}. This condition holds if for every pullback square
        \begin{displaymath}
            \begin{tikzcd}
                L \ar[r, "g"] \ar[d, "m"] & K \ar[d, "n"] \\
                J \ar[r, "f"] & I 
            \end{tikzcd}
        \end{displaymath}
        in ${\bf Set}$, the square
        \begin{equation} \label{BCforforall}
            \begin{tikzcd}
            PL  \ar[r, "\forall_g"]  & PK \\
            PJ \ar[r, "\forall_f"] \ar[u, "Pm"] & PI \ar[u, "Pn"]
            \end{tikzcd}
        \end{equation}
        commutes in ${\bf preHeyt}$ up to natural isomorphism.
        \item There is a some set $\Sigma$ and some element $\sigma \in P\Sigma$ such that for any set $I$ and any element $\varphi \in PI$ there exists a function $[\varphi]: I \to \Sigma$ such that $\varphi \cong P[\varphi](\sigma)$. Such an element $\sigma$ is called a \emph{generic element}.
    \end{enumerate}
\end{defi}

\begin{rema}{BCforexistential} In diagram (\ref{BCforforall}), we have stated the Beck-Chevalley condition for the universal quantifier; by taking the left adjoints of all the functors in this diagram, one can show that the corresponding Beck-Chevalley condition for the existential quantifier follows from that for the universal quantifier (indeed, they are equivalent).
\end{rema}

\begin{prop}{preHeytingalgfromimpalg}
    Let $A = (A, \cleq, \to, S)$ be an implicative algebra. If we preorder $A$ as follows:
    \[ a \vdash b :\Longleftrightarrow a \to b \in S, \]
    then $(A, \vdash)$ carries the structure of a preHeyting algebra, with $\to$ being the Heyting implication in this preHeyting algebra.
\end{prop}
\begin{proof}
    See \cite[Proposition 3.21]{miquel20}. We will prove a more general statement in \refprop{Heytconstr} below.
\end{proof}

\begin{rema}{logicalordering}
    We refer to $\vdash$ as the \emph{logical ordering} on $A$, as opposed to the evidential ordering $\cleq$. As mentioned before, the logical ordering is a derived, and weaker, concept. It is weaker because one can show that ${\bf i} = \bigcurlywedge_{a \in A} a \to a \in S$ (see \refcoro{othercombinators} below). Therefore $a \to a \in S$ for any $a \in A$, and, because $\to$ is monotone in the second argument, we have that $a \cleq b$ implies $a \to a \cleq a \to b \in S$.
\end{rema}

If $A = (A, \cleq, \to)$ is an implicative structure and $I$ is a set, then we can give $A^I$, the set of functions from $I$ to $A$, an implicative structure by choosing the pointwise ordering and implication. If $S$ is a separator on $A$, then we can define two separators on $A^I$:
\begin{eqnarray*}
    \varphi: I \to A \in S_I & :\Leftrightarrow & (\forall i \in I) \, \varphi(i) \in S \\
    \varphi: I \to A \in S^I & :\Leftrightarrow & \bigcurlywedge_{i \in I} \varphi(i) \in S
\end{eqnarray*}
Clearly, $S^I \subseteq S_I$, since separators are upwards closed. For our purposes, the first separator $S_I$ plays no role, while the second separator $S^I$ will be the relevant one. Indeed, this is the one which is used in the following fundamental result about implicative algebras.

\begin{theo}{imptripos} {\rm (Implicative tripos)}
Let $A = (A, \cleq, \to, S)$ be an implicative algebra. If to each set $I$ we assign the preHeyting algebra $PI = (A^I, \vdash_{S^I})$, while to each function $f: J \to I$ we assign the mapping $PI \to PJ$ sending a function $\varphi: I \to A$ to the composition $\varphi \circ f: J \to A$, then $P$ is a functor $P: {\bf Set}^{\rm op} \to {\bf preHeyt}$. Indeed, $P$ is a tripos whose generic element is the identity function $1_A: A \to A \in PA$.
\end{theo}
\begin{proof}
    See \cite[Theorem 4.11]{miquel20}. We will prove a more general statement in \reftheo{arrowtripos} below.
\end{proof}

\section{Arrow algebras}

As we mentioned in the introduction, we have two reasons for wanting to generalise implicative algebras. The first is to have more examples and the second is to ensure better closure properties, in particular, closure under subalgebras given through nuclei. We will show in this section that by generalising implicative algebras to arrow algebras we can address both these points.

The notion of an arrow algebra is obtained from that of an implicative algebra by modifying the second axiom for an implicative structure, which says that:
\[ a \to \bigcurlywedge_{b \in B} b = \bigcurlywedge_{b \in B} (a \to b), \]
if $a \in A, B \subseteq A$. Since 
\begin{equation} \label{theemptyone} 
    a \to \bigcurlywedge_{b \in B} b \cleq \bigcurlywedge_{b \in B} (a \to b) 
\end{equation}
holds as soon as $\to$ is monotone in the second argument (a property which we wish to preserve), this axiom is equivalent  to the inequality:
\begin{equation} \label{themeatyone} 
     \bigcurlywedge_{b \in B} (a \to b) \cleq a \to \bigcurlywedge_{b \in B} b.
\end{equation}
We will modify this axiom by making two changes to it:
\begin{enumerate}
    \item[(i)] We only require that the inequality (\ref{themeatyone}) holds for the logical ordering.
    \item[(ii)] We require that all the elements in $B$ are implications (that is, of the form $x \to y$ for some $x, y \in A$).
\end{enumerate}
So what we will do is that first we will define a new combinator:
\begin{eqnarray*} {\bf a} & := & \bigcurlywedge_{a \in A, B \subseteq {\rm Im}(\to)} (\bigcurlywedge_{b \in B} a \to b) \to a \to \bigcurlywedge_{b \in B} b \\
& = & \bigcurlywedge_{x,I,(y_i)_{i \in I}, (z_i)_{i \in I}} \big( \, \bigcurlywedge_{i \in I} x \to y_i \to z_i \, \big) \to x \to \big( \, \bigcurlywedge_{i \in I} y_i \to z_i \, \big) \end{eqnarray*}
where ${\rm Im}(\to)$ stands for the image of the implication operation $A \times A \to A$. Then we will drop the second axiom for an implicative structure and replace this with the requirement that ${\bf a} \in S$.

\begin{rema}{secondmodificationwhybelow}
    We will say more about our reasons for making the second modification (ii) in \refrema{simplerdefinitionnotOK} below.
\end{rema}

\subsection{Definition} These ideas lead to the following definitions.

\begin{defi}{arrowstr}
An \emph{arrow structure} is a complete poset $(A, \cleq)$ together with a binary operation $\to: A \times A \to A$ which is anti-monotone in the first and monotone in the second argument. That is, it satisfies the following condition:
        \begin{quote}
            If $a' \cleq a$ and $b \cleq b'$ then $a \to b \cleq a' \to b'$.
        \end{quote}
\end{defi}

\begin{defi}{separatorforarrowstr} Let $A = (A, \cleq, \to)$ be an arrow structure. A \emph{separator} on $A$ is a subset $S \subseteq A$ such that the following are satisfied:
        \begin{enumerate}
            \item[(1)] (Upwards closed) If $a \in S$ and $a \cleq b$, then $b \in S$.
            \item[(2)] (Closure under modus ponens) If $a \to b \in S$ and $a \in S$, then $b \in S$.
            \item[(3)] $S$ contains the combinators ${\bf k}$, ${\bf s}$ and ${\bf a}$.
        \end{enumerate}  
\end{defi}

\begin{defi}{arrowalg}
    An \emph{arrow algebra} is a quadruple $(A, \cleq, \to, S)$ where $(A, \cleq, \to)$ is an arrow structure and $S$ is a separator on that arrow structure.
\end{defi}

A fundamental property of arrow algebras is the following:

\begin{prop}{inttautinsep} Let $A$ be an arrow algebra with separator $S$. Suppose $\varphi$ is a propositional formula built from propositional variables $p_1,\dots, p_n$ and implications only. If $\varphi$ is an intuitionistic tautology, then $$\bigcurlywedge_{p_1,\dots,p_n\in A}\varphi\in S.$$
\end{prop}

The proof of this proposition relies on the following lemma, which also has many other useful applications and uses the assumption that the combinator $\bf a$ belongs to the separator.

\begin{lemm}{sepclosedunderlimitedMP} Suppose $A$ is an arrow algebra with separator $S$ and \[ (x_i)_{i \in I}, (y_i)_{i \in I}, (z_i)_{i \in I} \] are $I$-indexed families of elements from $A$. If
\begin{quote}
    $\bigcurlywedge_{i\in I} x_i\to y_i\to z_i\in S$ and $\bigcurlywedge_{i\in I} x_i\in S$,
\end{quote} 
then also $\bigcurlywedge_{i\in I}y_i\to z_i\in S$.
\end{lemm}
\begin{proof} Suppose $\bigcurlywedge_{i\in I} x_i\to y_i\to z_i\in S$ and $\bigcurlywedge_{i\in I} x_i\in S$. Because $\bigcurlywedge_{i\in I}x_i\preccurlyeq x_i$ for each $i\in I$ and separators are upwards closed, the first assumption implies that $$\bigcurlywedge_{i\in I}(\bigcurlywedge_{i\in I}x_i)\to y_i\to z_i\in S.$$  In addition, we have $$\mathbf{a}\preccurlyeq\big(\bigcurlywedge_{i\in I}(\bigcurlywedge_{i\in I}x_i)\to y_i\to z_i\big)\to (\bigcurlywedge_{i\in I}x_i)\to \bigcurlywedge_{i\in I}y_i\to z_i\in S.$$ Therefore we obtain the desired result by using that separators are closed under modus ponens.
\end{proof}

\begin{proof} (Of \refprop{inttautinsep}.)
    It is well-known that the axioms $A\to (B\to A)$ and $(A\to B\to C)\to (A\to B)\to (A\to C)$ together with modus ponens are complete for the implicative fragment of intuitionistic propositional logic (see, for instance, \cite[Theorem 5.1.10]{sorensenurzyczyn06}).

    This means that we can prove the statement in \refprop{inttautinsep} by induction on the number of applications of modus ponens used in the derivation of the tautology $\varphi$. Note that no propositional variable is a tautology, and therefore any tautology must be an implication $\varphi=\alpha\to\beta$.

    For the base case we are done by observing that $\mathbf{k},\mathbf{s}\in S$.
    
    So suppose $\varphi=\alpha\to\beta$ was constructed from the tautologies $\psi$ and $\psi\to\alpha\to\beta$ using the Modus Ponens rule. The induction hypothesis gives us that $\bigcurlywedge_{\overline{p}\in A}\psi,\bigcurlywedge_{\overline{p}\in A}\psi\to\alpha\to\beta\in S$. Therefore the desired result follows from \reflemm{sepclosedunderlimitedMP}.
\end{proof}

\begin{coro}{othercombinators}
    Let $S$ be a separator in an arrow algebra $A$. Then the following combinators belong to $S$:
    \begin{eqnarray*}
        {\bf i} & := & \bigcurlywedge_{a \in A} a \to a \\
        {\bf b} & := & \bigcurlywedge_{a,b,c \in A} (b \to c) \to (a \to b) \to (a \to c)
    \end{eqnarray*}
\end{coro}

Many arrow algebras satisfy the following additional property. 

\begin{defi}{compwithjoins} Let $(A, \cleq, \to)$ be an arrow structure. If for any $a \in A$ and $B \subseteq A$, the equation
        \[ (\bigcurlyvee_{b \in B} b) \to a = \bigcurlywedge_{b \in B} (b \to a) \]
holds, we say that the arrow structure is \emph{compatible with joins}. An arrow algebra is \emph{compatible with joins} if its underlying arrow structure is.
\end{defi}

Although the theory of arrow algebras works perfectly well without assuming that they are compatible with joins, it is still worthwhile to mention this property because it is satisfied in many examples (such as locales) and it can occasionally lead to some simplifications (see \reflemm{CompExt} below).

\subsection{Examples} Let us now give some examples of arrow algebras.

\subsubsection{Implicative algebras} Every implicative algebra is also an arrow algebra. The reason is that in an implicative algebra the inequality (\ref{themeatyone}) holds and therefore we have that:
\[ {\bf i} = \bigcurlywedge_{c \in A} c \to c \cleq \bigcurlywedge_{c, d \in A, c \cleq d} c \to d \cleq \bigcurlywedge_{a \in A, B \subseteq {\rm Im}(\to)} (\bigcurlywedge_{b \in B} a \to b) \to a \to \bigcurlywedge_{b \in B} b = {\bf a} \in S. \]
That is to say that for implicative structures any separator in the sense of \refdefi{separator} also contains the ${\bf a}$-combinator; therefore \refdefi{separator} is consistent with \refdefi{separatorforarrowstr} and any implicative algebra is also an arrow algebra.

\subsubsection{Pcas} Since every implicative algebra is an arrow algebra, we know that the downsets in a total combinatory algebra form an arrow algebra. However, we are now also able to cover the partial case.

\begin{theo}{arrowalgfrompcas}
    Let $(P, \leq, \cdot, P^\#)$ be a partial combinatory algebra. Then the collection $DP$ of downwards closed subsets of $P$ ordered by inclusion can be equipped with an arrow structure as follows:
    \[ X \to Y := \{ z \in P \, : \, (\forall x \in X) \, zx \downarrow \mbox{ and } zx \in Y \, \}. \]
    The downsets $X \in DP$ containing an element from the filter form a separator on this arrow structure and hence we obtain an arrow algebra. This arrow algebra is compatible with joins.
\end{theo}
\begin{proof}
    That $\to$ defines an arrow structure which is compatible with joins is easy to see; therefore we only show that
    \[ S := \{ X \in DP \, : \, (\exists x \in X) \, x \in P^\# \} \]
    defines a separator.

    \begin{enumerate}
        \item $S$ is clearly upwards closed.
        \item $S$ is closed under modus ponens: Suppose $X\to Y, X\in S$. Then there exist $x,z\in P^\#$ with $x\in X$ and $z\in X\to Y$. By the definition of implication, $zx \downarrow$ and $zx\in Y$ and by the definition of a filter, $zx \in P^\#$. Hence $Y\in S$.
        \item So it remains to show that the combinators ${\bf k}, {\bf s}, {\bf a} \in S$. Let $\underline{k}, \underline{s} \in P^\#$ be the combinators from the pca $P$ as in \refdefi{pca}. Then $\underline{k} \in {\bf k}$, because:
        \begin{eqnarray*}
            \mathbf{k} & = & \bigcap_{X,Y\in DP} X\to (Y\to X)\\
            & = & \bigcap_{X,Y\in DP}\{z\in P\mid \forall x\in X (zx\downarrow\text{and }zx\in Y\to X)\}\\
            & = & \bigcap_{X,Y\in DP}\{z\in P\mid \forall x\in X \, \big( \, zx\downarrow\text{and } \forall y \in Y \, ( \, zxy \downarrow \text{and } zxy \in  X \, ) \, \big) \, \} \\ & \ni &
            \underline{k}
        \end{eqnarray*}
    One shows $\underline{s}\in \mathbf{s}$ in a similar fashion.

    To prove $\mathbf{a}\in S$ we use \reftheo{combcompleteness}. We take \[ t:=\lambda^* uv, w.uvw\in P^\# \] and claim $t\in \mathbf{a}$. This holds iff for all $I, X,(Y_i)_{i \in I}, (Z_i)_{i \in I }\in DP$ and for all $s\in\bigcap_{i}X\to Y_i\to Z_i$, we have $ts\downarrow$ and $ts\in X\to \bigcap_i Y_i\to Z_i$.

Let $s\in\bigcap_i X\to Y_i\to Z_i$ and $x\in X$. We have $ts\downarrow$ and $tsx\downarrow$. Moreover, for any $i\in I$ and $y\in Y_i$ we have $sxy\downarrow$, so $tsxy\leq sxy\in Z_i$. Because $Z_i$ is downwards closed, we have $tsxy\in Z_i$ and we are done.
\end{enumerate}
\end{proof}

\subsubsection{PERs} The downset construction we saw above is the simplest way of constructing an arrow algebra out of a pca. However, there are many more. We will discuss one more example here, that of PERs (partial equivalence relations) over a pca $P$. 

If $P$ is a partial applicative poset, then so is $P \times P$ and therefore $D(P \times P)$ is an arrow structure. More explicitly, the elements of $D(P \times P)$ are downwards closed binary relations on $P$ which are ordered by inclusion and carry the following implication:
$$R\to S:=\{\, (x,x')\in P^2 \, \mid \, \forall(y,y')\in R \,\big( \, xy\downarrow, x'y'\downarrow, (xy,x'y')\in S \, \big) \, \}.$$

\begin{theo}{PERsasanarrowalgebra}
Let $(P,\cdot,\leq)$ be a partial applicative poset and write $${\rm PER}(P):=\{ \, R\in D(P\times P) \, \mid \, R\text{ is symmetric and transitive} \, \}$$ for the set of all binary relations on $P$ which are downwards closed, symmetric and transitive. If we order ${\rm PER}(P)$ by inclusion and equip it with the implication defined above, we obtain a arrow structure which is compatible with joins. Moreover if $(P,P^\#,\cdot,\leq)$ is a pca, defining $$S:=\{ \, R\in{\rm PER}(P) \, \mid \, \exists (x,y)\in R \, \big( \, x,y\in P^\# \, \big) \, \}$$ gives us a separator on this arrow structure.
\end{theo}
\begin{proof}
    One readily verifies that the elements in ${\rm PER}(P)$ are closed under implication and arbitrary intersections, so we can conclude that ${\rm PER}(P)$ does indeed carry an arrow structure.

    To see that it is compatible with joins, we first observe that the join is given by the transitive closure of the union. 
    We only need to show $$(\bigcurlyvee_{i\in I} R_i)\to S\succcurlyeq\bigcurlywedge_{i\in I}(R_i\to S)$$ as the other inequality is always true. So let $(x,x')\in \bigcurlywedge_{i\in I}(R_i\to S)$ and $(a,a')\in\bigcurlyvee_{i\in I} R_i$. The latter implies that there are elements $a_0,a_1,\dots, a_n\in A$ such that $(a_k,a_{k+1})\in R_{i_k}$ for $i_k\in I$, $k \lt n$ and additionally $a_0=a$ and $a_n=a'$. By symmetry and transitivity $(x',x')\in \bigcurlywedge_{i\in I}(R_i\to S)$, so $(xa_0,x'a_1),(x'a_k,x'a_{k+1})\in S$. By transitivity $(xa,x'a')=(xa_0,x'a_n)\in S$.
    
    It remains to prove that $S$ is a separator. This can be shown by an argument similar to that in the proof of \reftheo{arrowalgfrompcas} with $(\underline{k},\underline{k})\in\mathbf{k}$, $(\underline{s},\underline{s})\in\mathbf{s}$ and $(t,t)\in\mathbf{a}$.
\end{proof}

\subsection{Nuclei} Another advantage of arrow algebras is that they are closed under subalgebras, where these subalgebras are determined by \emph{nuclei}, in a manner which is a direct generalisation of the theory of nuclei in locale theory.

\begin{defi}{nucleus}
    Let $A = (A, \cleq, \to, S)$ be an arrow algebra. A mapping $j: A \to A$ will be called a \emph{nucleus} if the following three properties are satisfied:
            \begin{enumerate}
                \item[(1)] $a \cleq b$ implies $ja \cleq jb$ for all $a, b \in A$.
                \item[(2)] $\bigcurlywedge_{a \in A} a \to ja \in S$.
                \item[(3)] $\bigcurlywedge_{a, b \in A} (a \to jb) \to (ja \to jb) \in S$.
            \end{enumerate}   
\end{defi}

\begin{exam}{examofnuclei}
    Let $A$ be an arrow algebra and $c \in A$. It follows from \refprop{inttautinsep} that the following define nuclei on $A$:
    \begin{enumerate}
        \item[(i)] $jx = c \to x$.
        \item[(ii)] $jx = (x \to c) \to c$.
        \item[(iii)] $jx = (x \to c) \to x$. 
    \end{enumerate}  
    (See \cite{vandenberg19} for more information.)
\end{exam}

\begin{lemm}{nuclproperties}
    Let $j$ be a nucleus on an arrow algebra $A$ with separator $S$. Then we have the following:
    \begin{enumerate}
        \item[(i)] $\bigcurlywedge_{a\in A} jja\to ja\in S$.
        \item[(ii)] $\bigcurlywedge_{a,b\in A} (a\to b)\to (ja\to jb)\in S$.
        \item[(iii)] $\bigcurlywedge_{a,b\in A}j(a\to b)\to (ja\to jb)\in S$.
    \end{enumerate}
\end{lemm}
\begin{proof} Recall that we have $\mathbf{i}, \mathbf{b} \in S$, where $\mathbf{i}$ and $\mathbf{b}$ were defined in \refcoro{othercombinators}.
    \begin{enumerate}
        \item[(i)] We have $\bigcurlywedge_a ja\to ja\in S$, because $\mathbf{i} \in S$, and $\bigcurlywedge_a (ja\to ja)\to (jja\to ja)\in S$ by the definition of a nucleus. So $\bigcurlywedge_{a}jja\to ja\in S$ follows from \reflemm{sepclosedunderlimitedMP}.
        
        \item[(ii)] Because $\mathbf{b}\in S$ we get \begin{align*} & \bigcurlywedge_{a,b\in A}\big( \, (a\to b)\to(a\to jb) \, \big)\to \big( \, (a\to jb)\to (ja\to jb) \, \big) \\& \to\big( \, (a\to b)\to (ja\to jb) \, \big)\in S.\end{align*}
        Therefore $\bigcurlywedge(a\to b)\to (ja\to jb)\in S$ will follow as soon as we can show that $\bigcurlywedge (a\to b)\to (a\to jb)\in S$ and $\bigcurlywedge(a\to jb)\to (ja\to jb)\in S$. The latter is true by definition of a nucleus. For the former we note that $\bigcurlywedge_{a,b} (b \to jb) \to (a\to b)\to (a\to jb)\in S$ by \refprop{inttautinsep} and $\bigcurlywedge_{b \in B} b \to jb \in S$ by the definition of a nucleus. Therefore the desired result follows from \reflemm{sepclosedunderlimitedMP}.

        \item[(iii)] We repeatedly apply \refprop{inttautinsep} and \reflemm{sepclosedunderlimitedMP}. From the conclusion of (ii)$$\bigcurlywedge (a\to b)\to (ja\to jb)\in S,$$ we obtain $$
        \bigcurlywedge ja\to (a\to b)\to jb\in S.$$Using the definition of a nucleus, this implies $$\bigcurlywedge ja \to j(a\to b)\to jb\in S$$ and hence $$ \bigcurlywedge j(a\to b)\to (ja\to jb)\in S. $$
    \end{enumerate}
\end{proof}

\begin{rema}{intuitionisticreasoning} In the sequel we will often say that something follows by ``intuitionistic reasoning'' if it can be deduced from the combination of \refprop{inttautinsep} and \reflemm{sepclosedunderlimitedMP}.
\end{rema}

In locale theory we would obtain a new locale (a sublocale) by restricting to the elements $a \in A$ for which $ja = a$. For arrow algebras we have the following construction.

\begin{prop}{subarrowalgebras}
Let $(A, \cleq, \to, S)$ be an arrow algebra and $j: A \to A$ be a nucleus on it. Then $A_j = (A, \cleq, \to_j, S_j)$ with
            \begin{eqnarray*}
                a \to_j b & :\equiv & a \to jb \\
                a \in S_j & :\Leftrightarrow & ja \in S
            \end{eqnarray*}
            is also an arrow algebra. If $A$ is compatible with joins, then so is $A_j$.
\end{prop}

\begin{proof}
        We have to show that:
        \begin{enumerate}
            \item $\to_j$ defines an arrow structure.
            \item $S_j$ is upwards closed.
            \item $S_j$ is closed under modus ponens.
            \item $S_j$ contains the combinators.
        \end{enumerate}

        (1) and (2) follow from the monotonicity of $j$.

        In order to show (3), suppose $a\to_j b\in S_j$ and $a\in S_j$; that is, $j(a\to jb)\in S$ and $ja\in S$. Point (iii) of \reflemm{nuclproperties} implies $ja\to jjb\in S$ and the closure of $S$ under modus ponens implies $jjb\in S$. From (i) of the same lemma we get $jb\in S$, and thus $b\in S_j$.
        
        (4) To show that $S_j$ contains the combinators, note that $S \subseteq S_j$ and hence it suffices to prove that $S$ contains the combinators.

        First of all, consider $$\mathbf{k}_j:=\bigcurlywedge_{a,b}a\to j(b\to ja).$$ From $\mathbf{k}\in S$, the laws for a nucleus and ``intuitionistic reasoning'', it follows that $\mathbf{k}_j\in S$; hence $\mathbf{k}_j\in S_j$ holds as well.

        Next, consider \begin{align*}
            \mathbf{s}_j&:=\bigcurlywedge_{a,b,c}(a\to_j b\to_j c)\to_j (a\to_j b)\to_j a\to_j c.
        \end{align*}
     From the properties of a nucleus, we obtain $$(a\to j(b\to jc))\to (a\to jb\to jc)\in S.$$ Since also $\mathbf{s}\in S$, we obtain $\mathbf{s}_j\in S$ and hence $\mathbf{s}_j\in S_j$ as well.

    Finally, consider \begin{align*}
        \mathbf{a}_j:=&\bigcurlywedge_{x, I,(y_i)_{i \in I},(z_i)_{i \in I}}(\bigcurlywedge_{i\in I}x\to_j y_i\to_j z_i)\to_j x\to_j(\bigcurlywedge_{i\in I}y_i\to_j z_i)\\
        =&\bigcurlywedge_{x, I,(y_i)_{i \in I},(z_i)_{i \in I}}\big(\bigcurlywedge_{i\in I}x\to j( y_i\to j z_i)\big)\to j\big( x\to j(\bigcurlywedge_{i\in I}y_i\to j z_i)\big)
    \end{align*}
    We will again show that $\mathbf{a}_j \in S$, from which the desired $\mathbf{a}_j \in S_j$ follows.

    By repeatedly applying \reflemm{nuclproperties}, we see that $\mathbf{a}_j \in S$ follows from 
    \[ \bigcurlywedge_{x,I, (y_i)_{i \in I},(z_i)_{i \in I}}(\bigcurlywedge_{i\in I}x\to jy_i\to jz_i)\to x\to \bigcurlywedge_{i\in I} jy_i\to jz_i\in S. \]
    But because $\mathbf{a} \in S$ and \begin{align*}
        \mathbf{a}\preccurlyeq&\bigcurlywedge_{x, I, (y_i)_{i \in I},(z_i)_{i \in I}}(\bigcurlywedge_{i\in I}x\to jy_i\to jz_i)\to x\to \bigcurlywedge_{i\in I} jy_i\to jz_i\in S,
    \end{align*}
    the desired result follows.
\end{proof}

\begin{rema}{simplerdefinitionnotOK} We are now in a position to address the question why we are making not just the first modification (i), but also the second modification (ii) to the definition of an implicative algebra (see the discussion preceding \refdefi{arrowstr}). Indeed, let us make the following definition:
\begin{defi}{strongarrowalgebra} An arrow algebra $(A, \cleq, \to, S)$ will be called \emph{strong} if $$\mathbf{a}':=\bigcurlywedge_{a \in A, B \subseteq A}(\bigcurlywedge_{b \in B }a\to b)\to a\to\bigcurlywedge_{b\in B }b \in S. $$ 
\end{defi}
So the question is why we are not requiring every arrow algebra to be strong. 

The main reason is the following.
\begin{prop}{notstrongarrowalgebra}
    The arrow algebra $D(K_1)$ obtained from the pca $K_1$, as in the section on pcas, is not strong.
\end{prop}
\begin{proof}
    Suppose $t' \in {\bf a}'$ and let $x \in K_1$. By choosing $B = \emptyset$ and $a = \top$, we see that $t'x$ is always defined and a total function $\mathbb{N} \to \mathbb{N}$. On the other hand, suppose that $xy$ is defined and $xy = z$. Then by choosing $a = \{ y \}$ and $B = \{ \{ z \} \}$, we see that $t'xy$ is also defined and $t'xy = z$. In other words, $t'x$ extends $x$ to a total function. 
    
    This is not possible in $K_1$, due to the undecidability of the Halting Problem. Indeed, let $x$ be a code for the partial recursive function that on $n$ outputs the number of steps needed for $n \cdot n$ to terminate, and which will be undefined if $n \cdot n$ does not terminate. Then we use $t'xn$ to decide the Halting Problem, as follows: we run $n \cdot n$ up to $t'xn$ steps and say that $n \cdot n$ will never terminate if $n \cdot n$ has not yet terminated at that point. 
\end{proof}

The second reason is that in the proof of the preceding proposition, \refprop{subarrowalgebras}, we are unable to show that ${\bf a}'_j \in S_j$ if ${\bf a}' \in S$. For this, we do not have a counterexample; we are only saying that we have been unable to prove that it is true.
\end{rema}

\section{The interpretation of $\lambda$-terms in strong arrow algebras}

In his work Miquel shows that $\lambda$-terms can be interpreted in any implicative algebra (see \cite[Proposition 3.4]{miquel20}). In this section we will show that this is true more generally for \emph{strong} arrow algebras, using ideas very similar to Miquel's (indeed, note that every implicative algebra is a strong arrow algebra). In the process of showing this we will also prove that strong arrow algebras can be equipped with the structure of \emph{implicative ordered combinatory algebra} in the sense of \cite{ferrersantosetal17}: this will prove helpful in showing that arrow algebras give rise to triposes, as we will see in the next section.

\subsection{The main result} We remind the reader that $\lambda$-terms are defined as follows.

\begin{defi}{lambdaterms}
    The collection of \emph{$\lambda$-terms} is the smallest collection of formal expressions satisfying the following:
    \begin{enumerate}
        \item Each variable is a $\lambda$-term.
        \item (Application) If $M$ and $N$ are $\lambda$-terms, then so is $MN$.
        \item (Abstraction) If $M$ is a $\lambda$-term and $x$ is a variable, then $\lambda x.M$ is a $\lambda$-term.
    \end{enumerate}
\end{defi}

In order to interpret $\lambda$-terms, we need to interpret both application and abstraction. We start with the former. In what follows $A$ will be an arbitrary arrow structure.

\begin{defi}{applicationinarrowstr} (Application) For all $a,b\in A$ we define $$U_{a,b}:=\{x \in A\mid a\preccurlyeq b\rightarrow x \}$$ and the \emph{application} $$ab:=\bigcurlywedge U_{a,b}$$
\end{defi}

\begin{lemm}{AppProp}
    For any arrow structure, the application has the following properties:
    \begin{enumerate}
        \item If $a\preccurlyeq a'$ and $b\preccurlyeq b'$, then $ab\preccurlyeq a'b'$.
        \item $(a \to b)a \cleq b$.
    \end{enumerate}
\end{lemm}
    \begin{proof}
        (1) Because the implication is antitone in its first operand we get $U_{a',b'}\subseteq U_{a,b}$ whenever $a\preccurlyeq a'$ and $b\preccurlyeq b'$.

        (2) Note $b \in U_{a\to b, a}$. Hence $(a \to b)a = \bigcurlywedge U_{a\to b,a}\preccurlyeq b$.
    \end{proof}

Next, we define abstraction for arrow structures.

\begin{defi}{abstrinarrowstr} (Abstraction)
    Let $A$ be an arrow structure and $f:A\to A$ a function. Then we define $$\lambda f:=\bigcurlywedge_{x\in A}x\to f(x)\in A.$$
\end{defi}

\begin{rema}{notation} At this point we will fix some important notational conventions. Point (1) of \reflemm{AppProp} implies that each arrow structure can be considered as a \emph{total} applicative poset (see \refdefi{pap}). For arrow structures we use the same notational conventions as for applicative posets; in particular, we will use the convention that application associates to the left, meaning that $abc$ stands for $(ab)c$.

To further reduce clutter we will say that in the following list the operations to the left bind stronger than the operation to the right:
\[ \cdot \mbox{ (application)} \quad \to  \quad \bigcurlywedge \] 
So, for instance, $a\to bc$ is shorthand for $a\to (bc)$.
\end{rema}

\begin{lemm}{abstraction}
    If $f,g:A\to A$ are functions, then we have the following properties:\begin{enumerate}
        \item If $f(a)\preccurlyeq g(a)$ for all $a\in A$, then $\lambda f\preccurlyeq \lambda g$.
        \item $(\lambda f)a\preccurlyeq f(a)$ for all $a\in A$.
    \end{enumerate}
\end{lemm}
\begin{proof}
    (1) By monotonicity in the second component of the implication we see $$\lambda f=\bigcurlywedge_{a}a\to f(a)\preccurlyeq\bigcurlywedge_{a}a\to g(a)=\lambda g.$$
    (2) $(\lambda f)a = (\bigcurlywedge_x x\to f(x))a\preccurlyeq (a\to f(a))a\preccurlyeq f(a)$ by \reflemm{AppProp}(2).
\end{proof}

Now that we have defined an interpretation for application and abstraction, we can recursively define an interpretation in an arrow structure for each $\lambda$-term.

\begin{defi}{lambdainterpretation}
    Let $A$ be an arrow structure and $M$ be a $\lambda$-term with free variables among $x_1,\dots,x_n$. The \emph{interpretation} $M^A: A^n\to A$ of the $\lambda$-term $M$ in $A$ will recursively be defined as\begin{enumerate}
        \item If $M=x_i$, then $M^A$ is the projection on the $i$th coordinate.
        \item For products: $(MN)^A:A^n\to A:\overline{a}\mapsto M^A(\overline{a})N^A(\overline{a})$.
        \item If the free variables of $M$ are among $x_1,\dots,x_n,x$, then $$(\lambda x.M)^A:\overline{a}\mapsto \lambda(a\mapsto M(\overline{a},a))$$
    \end{enumerate}
\end{defi}

\begin{prop}{monotonicityoflambdaterms}
    If $M$ is a $\lambda$-term with free variables among $x_1,\dots, x_n$ and $a_1\preccurlyeq b_1,\dots, a_n\preccurlyeq b_n$ in an arrow structure $A$, then $$M^A(\overline{a})\preccurlyeq M^A(\overline{b})$$
\end{prop}
\begin{proof}
    By induction on the structure of $M$ using \reflemm{AppProp}(1) and \reflemm{abstraction}(1).
\end{proof}

We can now state this section's main result.

\begin{theo}{lambda}
    Let $A$ be any strong arrow algebra with separator $S$. If $M$ is a $\lambda$-term with free variables among $x_1,\dots,x_n$ and $a_1,\dots,a_n\in S$, then $$M^A(a_1,\dots, a_n)\in S.$$ In particular, $M^A \in S$ if $M$ is closed.
\end{theo}

\subsection{The proof} To prove the main result, we are going to need some auxiliary lemmas. The idea is that for each closed $\lambda$-term $M$, we give an associated `combinatory term' $M_*$ with $M_*^A\preccurlyeq M^A$ and $M_*^A \in S$. So the combinatory term gives enough `evidence' for the original term to be an element of the separator.

We will start with proving some properties of the separator of a strong arrow algebra.
\begin{lemm}{SepApp}
    Separators in strong arrow algebras are closed under application.
\end{lemm}
    \begin{proof}
        Let $S$ be a separator and suppose $a,b\in S$. Because 
        $$a\preccurlyeq\bigcurlywedge_{x \in U_{a,b}}b\to x$$ 
        we have $\bigcurlywedge_{x \in U_{a,b}}b\to x \in S$. Now $\mathbf{a}' \in S$ implies
        \begin{align*} \big( \, \bigcurlywedge_{ x \in U_{a,b}}b\to x \, \big) \to b \to ab  & = \\
            \big( \, \bigcurlywedge_{x \in U_{a,b}}b\to x \, \big) \to b \to \bigcurlywedge_{x\in U_{a,b}} x & \in S. \end{align*} Using that separators are closed under modus ponens (twice), we obtain that $ab \in  S$.
    \end{proof}

\begin{lemm}{etainsep}
    If $S$ is a separator in a strong arrow algebra, then $$\Em:=\bigcurlywedge_{a \in A} a\to\bigcurlywedge_{b \in A} b\to ab\in S.$$.
\end{lemm}
    \begin{proof} By the definition of $U_{a,b}$ we have
        \begin{align*}
            \mathbf{i}=\bigcurlywedge_{a \in A} a\to a&\preccurlyeq \bigcurlywedge_{a \in A}  a\to \bigcurlywedge_{b \in A} \bigcurlywedge_{x \in U_{a,b}}b\to x\\
            &\preccurlyeq\bigcurlywedge_{a,b \in A}a\to\bigcurlywedge_{x \in U_{a,b}}b\to  x \in S.
        \end{align*} Here the first inequality follows from the definition of $U_{a,b}$, while the second holds because $\to$ is monotone in the second argument (as in equation (\ref{theemptyone})).
        
        Similarly, as in the previous lemma, we have that \[  \bigcurlywedge_{a,b \in A} \, \big( \, \bigcurlywedge_{x \in U_{a,b}}b\to x \, \big) \to b \to ab \in S. \]
        Therefore it follows from \reflemm{sepclosedunderlimitedMP} that $$\bigcurlywedge_{a,b \in A}a\to b\to ab\in S.$$
        Using $\mathbf{a}\in S$, again in combination with \reflemm{sepclosedunderlimitedMP}, we will get $\Em \in S$.
    \end{proof}

\begin{prop}{arrowalgebratca} If $A$ is a strong arrow algebra with $\cdot$ the application from \refdefi{applicationinarrowstr}, then $(A, \cdot, \cleq, S)$ is a total combinatory algebra. In particular, for the following elements 
    \begin{eqnarray*}
        \Km & := & \bigcurlywedge_{a\in A} a\to\bigcurlywedge_{b\in A} b\to a\\
        \Sm & := & \bigcurlywedge_{a\in A} a\to \bigcurlywedge_{b\in A} b\to \bigcurlywedge_{c \in A} c \to ac(bc)
    \end{eqnarray*}
    we have $\Km, \Sm \in S$, $\Km ab \cleq a$ and $\Sm abc \cleq ac(bc)$ for all $a, b, c \in A$.
\end{prop}
    \begin{proof}
    It suffices to prove that $\Km, \Sm \in S$. That $\Km \in S$ follows from \reflemm{sepclosedunderlimitedMP} in combination with $\km, \am \in S$.
    
    From $\sm \in S$ it follows that 
    \[ \bigcurlywedge_{a,b,c \in A} (c \to bc \to ac(bc)) \to (c \to bc) \to c \to ac(bc) \in S. \]
    In addition, the previous lemma implies that
    \[ \bigcurlywedge_{b,c \in A} b \to c \to bc, \bigcurlywedge_{a,b,c \in A} ac \to bc \to ac(bc), \bigcurlywedge_{a,c \in A} a \to c \to ac \in S. \]
    Therefore we obtain by intuitionistic reasoning that
    \[ \bigcurlywedge_{a,b,c \in A} a \to b \to c \to ac(bc) \in S. \] 
    Now $\Sm \in S$ follows from this in combination with $\am \in S$.
\end{proof}

\begin{defi}{implicativeoca} 
    \cite[Definition 3.7]{ferrersantosetal17} A total combinatory algebra $(P, \cdot, \leq, P^\#)$ is an \emph{implicative ordered combinatory algebra (implicative oca)} if:
    \begin{enumerate}
        \item[(i)] $(P, \leq)$ is a complete poset.
        \item[(ii)] there is an operation $\to: P^{\rm op} \times P \to P$ which is antitone in the first and monotone in the second component, such that
        \[ a \leq b \to c \Longrightarrow ab \leq c. \] 
        \item[(iii)] there is an element ${\bf E} \in P^\#$ such that
        \[ ab \leq c \Longrightarrow {\bf E}a \leq b \to c. \]
    \end{enumerate}
\end{defi}

\begin{coro}{strongalgebrasareimplicativeocas}
    Every strong arrow algebra can be equipped with the structure of an implicative ordered combinatory algebra.
\end{coro}
\begin{proof}
    This follows from the previous two results.
\end{proof}    

\begin{defi}{combinatoryterm} The collection of \emph{combinatory terms} is the smallest collection of formal expressions satisfying the following rules: 
\begin{enumerate}
    \item Every variable is a combinatory term.
    \item The combinators $\Idm, \Km, \Sm$ and $\Em$ are combinatory terms.
    \item If $M$ and $N$ are combinatory terms, then so is $MN$.
\end{enumerate}
\end{defi}

Like $\lambda$-terms, combinatory terms can be interpreted in any arrow structure.

\begin{defi}{intofcombterms}
    Let $A$ be an arrow structure and $M$ be a combinatory term with free variables among $\overline{x}=x_1,\dots,x_n$. For $\overline{a}=a_1,\dots,a_n\in A^n$, we define the interpretation $M^A[\overline{a}/\overline{x}]$ inductively as follows:\begin{enumerate}
        \item $x^A_i[\overline{a}/\overline{x}]:=a_i$
        \item $\Idm^A[\overline{a}/\overline{x}]:=\idm$
        \item $\Km^A[\overline{a}/\overline{x}]:= \Km$
        \item $\Sm^A[\overline{a}/\overline{x}]:=\Sm$
        \item $\Em^A[\overline{a}/\overline{x}]:=\Em$
        \item $(MN)^A[\overline{a}/\overline{x}]=M^A[\overline{a}/\overline{x}]N^A[\overline{a}/\overline{x}]$
    \end{enumerate} 
\end{defi}

\begin{lemm}{combsep} Suppose $A$ is a strong arrow algebra with separator $S$. If $M$ a combinatory term with variables among $\overline{x}$ and $\overline{a}\in S^n$, then $M^A[\overline{a}/\overline{x}]\in S$.
\end{lemm}
\begin{proof}
        Proof by induction on the structure of $M$, using \reflemm{SepApp}, \reflemm{etainsep} and \refprop{arrowalgebratca}.
    \end{proof}

    Next, we introduce a form of abstraction for the combinatory terms as well. 

\begin{defi}{lambdaforcombterms}
    For a combinatory term  $M$ and a variable $x$ we define the combinatory term $\langle x\rangle M$ inductively as follows:
    \begin{enumerate}
        \item $\langle x\rangle x:= \Idm$.
        \item $\langle x\rangle M:=\Km M$ if $M$ is $\Idm$, $\Km$, $\Sm$, $\Em$ or a variable different from $x$.
        \item $\langle x\rangle (MN):=\Sm(\langle x\rangle M)(\langle x\rangle N)$.
    \end{enumerate}
\end{defi}

With this definition we now proceed to prove that $\langle x\rangle M$ indeed behaves like an abstraction. 

\begin{lemm}{combabs}
    If $M$ is a combinatory term, $x$ a variable and the free variables of $\langle x\rangle M$ are contained in $\overline{y}$, then:
    \begin{enumerate}
        \item The free variables of $\langle x\rangle M$ are those of $M$ minus $x$.
        \item For any arrow structure $A$ and all $a\in A$, $\overline{a}\in A^n$, we have $$(\langle x\rangle M)^A[\overline{a}/\overline{y}]a\preccurlyeq M^A[\overline{a}/\overline{y},a/x]$$
    \end{enumerate}
\end{lemm}
    \begin{proof}
        This is a simple proof by induction on the structure of $M$ using the properties of $\Km$ and $\Sm$ as stated in \refprop{arrowalgebratca}.
    \end{proof}

\begin{defi}{translatinglambdatocomb}
    For each $\lambda$-term $M$ we define an associated combinatory term $M_*$ as follows:
    \begin{enumerate}
        \item For a variable $x$, we put $x_*:= x$.
        \item $(MN)_*:= M_* N_*$.
        \item $(\lambda x.M)_*:=\Em(\langle x\rangle M_*)$.
    \end{enumerate}
\end{defi}

\begin{lemm}{lambdacomb}
    Suppose $M$ is a $\lambda$-term with free variables among $\overline{y}=y_1,\dots, y_n$. For any arrow structure $A$ and $\overline{a}\in A^n$ we have $$M^A_*[\overline{a}/\overline{y}]\preccurlyeq M^A(\overline{a}).$$
\end{lemm}
    \begin{proof}
    Proof by induction on the structure of $M$. We will only show the abstraction case $M=\lambda x.N$ as the other two cases are trivial. \begin{eqnarray*}
    M_*^A[\overline{a}/\overline{y}] & = &(\Em(\langle x\rangle N_*))^A[\overline{a}/\overline{y}]\\
    &\preccurlyeq & \bigcurlywedge_{a\in A}a\to (\langle x\rangle N_*)^A[\overline{a}/\overline{y}]a\\
    &\preccurlyeq & \bigcurlywedge_{a\in A}a\to  N_*^A[\overline{a}/\overline{y},a/x]\\
    &\preccurlyeq & \bigcurlywedge_{a\in A} a\to N^A(\overline{a},a)\\
    &= & (\lambda x.N)^A(\overline{a})\\
    &= & M^A(\overline{a})
    \end{eqnarray*}
    where the second and third inequality follow from \reflemm{combabs} and the induction hypothesis respectively.
    \end{proof}

This finishes the proof of \reftheo{lambda}, as it now follows from \reflemm{lambdacomb} and \reflemm{combsep}. 

\begin{rema}{comparisonwithothers} The proof we have just given is heavily inspired by Miquel's work on implicative algebras. However, for implicative algebras we have the ``adjunction property''
\begin{equation} \label{adjunction} ab \cleq c \Leftrightarrow a \cleq b \to c, \end{equation}
which simplifies the proof of \reftheo{lambda}. Indeed, for strong arrow algebras it seems we can only prove the direction for left to right in (\ref{adjunction}) with an additional ${\bf E} \in S$. 

In an earlier version of this paper we proved a result like \reftheo{lambda} for general arrow algebras, but with a different interpretation of the application $ab$ and the abstraction $\lambda f$. The proof of this was quite ugly and it turned out that it could be avoided, so it is no longer included in the present version of the paper. (It can still be found in the MSc thesis of the second author \cite{briet23}.)
\end{rema}

\section{Arrow triposes}

We now turn to proving the main result about arrow algebras: that they give rise to triposes. That is, we will generalise \refprop{preHeytingalgfromimpalg} and \reftheo{imptripos} to arrow algebras. We will start by doing this for strong arrow algebras and then explain how to generalise this result to all arrow algebras.

\subsection{PreHeyting algebras from strong arrow algebras} Let us first explain how to construct a preHeyting algebra from a strong arrow algebra.

\begin{prop}{strongHeytconstr} 
    Let $A=(A,\preccurlyeq,\to,S)$ be a strong arrow algebra. In that case we can define a preorder $\vdash$ on $A$ as follows: 
    $$a\vdash b\text{ if }a\to b\in S.$$
    We refer to $\vdash$ as the \emph{logical ordering} and with this preorder, $(A,\vdash)$ is a preHeyting algebra with $\to$ as Heyting implication.
\end{prop}
\begin{proof} This follows from \cite[Theorem 4.15]{ferrersantosetal17} and \refcoro{strongalgebrasareimplicativeocas}. For the convenience of the reader and for future reference, we outline a proof here.
    
Note that $a \vdash b$ is equivalent to the existence of an $r \in S$ such that $ra \cleq b$. Indeed, if $a \to b \in S$ then $ra \cleq b$ for $r = a \to b$; conversely, if $ra \cleq b$ for some $r \in S$, then ${\bf E}r \cleq a \to b$ and hence $a \to b \in S$.

\emph{Preorder:} Reflexivity and transitivity follow from $\mathbf{i}\in S$ and $\mathbf{b}\in S$ respectively (see \refcoro{othercombinators}).

\emph{Bounded:} This follows from the general fact that $\vdash$ contains $\preccurlyeq$, which is true because $\mathbf{i}\in S$ and implication is monotone. In particular, we have $a\vdash\top$ and $\bot\vdash a$, because $a \cleq \top$ and $\bot \cleq a$.

\emph{Binary meets:} Let us define
\begin{eqnarray*}
    {\bf p} & := & (\lambda xyz.zxy)^A \\
    {\bf p}_0 & := & (\lambda x.x(\lambda xy.x))^A \\
    {\bf p}_1 & := & (\lambda x.x(\lambda xy.y))^A
\end{eqnarray*}
Then ${\bf p}, {\bf p}_0, {\bf p}_1 \in S$ by \reftheo{lambda}. For all $a, b \in A$ we have ${\bf p}ab \vdash a$ and ${\bf p}ab \vdash b$, because ${\bf p}_0({\bf p}ab) \cleq a$ and ${\bf p}_1({\bf p}ab) \cleq b$. Moreover, if $ra \cleq c$ and $sb \cleq c$ where $r, s \in S$, then $tc \cleq {\bf p}ab$, for
\[ t = (\lambda x.{\bf p}(rx)(sx))^A, \]
with $t \in S$ following from \reftheo{lambda}.
This shows that ${\bf p}ab$ acts as the meet $a \land b$ of $a$ and $b$.

\emph{Implication:} We have to show that if $r({\bf p}ab) \cleq c$ for some $r \in S$, then $sa \cleq b \to c$ for some $s \in S$, and conversely.

If $r \in S$ is given such that $r({\bf p}ab) \cleq c$, then:
\begin{displaymath}
    \begin{array}{rclc}
    r({\bf p}ab) & \cleq & c & \Longrightarrow \\
    (\lambda x.r({\bf p}ax))^Ab & \cleq & c & \Longrightarrow \\ 
    ({\bf E}(\lambda x.r({\bf p}ax)))^A & \cleq & b \to c & \Longrightarrow \\ 
    (\lambda y.{\bf E}(\lambda x.r({\bf p}yx)))^Aa & \cleq & b \to c,
\end{array}
\end{displaymath}
so we can choose $s = (\lambda y.{\bf E}(\lambda x.r({\bf p}yx)))^A \in S$, using \reftheo{lambda}.

Conversely, if $sa \cleq b \to c$, then $sab \cleq c$ and therefore $r({\bf p}ab) \cleq c$ for 
\[ r = \lambda x.s({\bf p}_0x)({\bf p}_1x) \in S. \]

\emph{Binary joins:} For elements $a,b\in A$ we define
\begin{eqnarray*}
    a \lor b & := & \bigcurlywedge_c (a\to c)\to (b\to c)\to c, \\
    {\bf i}_0 &:=& (\lambda zxy.zx)^A \in S, \\
    {\bf i}_1 &:=& (\lambda zxy.zy)^A \in S.
\end{eqnarray*}
Then we have 
\begin{eqnarray*}
    {\bf i}_0a & = & (\lambda zxy.xz)^Aa \\
    & \cleq & (\lambda xy.xa)^A \\
    & = & \bcw_{x} x \to \bcw_{y} y \to xa \\
    & \cleq & \bcw_{x,y} x \to y \to xa \\
    & \cleq & \bcw_c (a \to c) \to (b \to c) \to (a \to c)a \\
    & \cleq & \bcw_c (a \to c) \to (b \to c) \to c \\
    & = & a \lor b,
\end{eqnarray*}
and ${\bf i}_1b \cleq a \lor b$ by a similar argument.

If $a \to c \in S$ and $b \to c \in S$, then $r = (\lambda x.x(a \to c)(b \to c))^A \in S$ and 
\begin{eqnarray*}
    r(a \lor b) & \cleq & (a \lor b)(a \to c)(b \to c) \\
    & \cleq & ((a \to c) \to (b \to c) \to c)(a \to c)(b \to c) \\
    & \cleq & c,
\end{eqnarray*}
showing that $a \lor b$ is the join of $a$ and $b$ with respect to $\vdash$.
\end{proof}

\subsection{Triposes from strong arrow algebras}Similarly as for implicative algebras, we can note that if $A = (A, \cleq, \to)$ is an arrow structure and $I$ is a set, then we can give $A^I$, the set of functions from $I$ to $A$, an arrow structure by choosing the pointwise ordering and implication. If $S$ is a separator on $A$, then we can define two separators on $A^I$:
\begin{eqnarray*}
    \varphi: I \to A \in S_I & :\Leftrightarrow & (\forall i \in I) \, \varphi(i) \in S \\
    \varphi: I \to A \in S^I & :\Leftrightarrow & \bigcurlywedge_{i \in I} \varphi(i) \in S
\end{eqnarray*}
Clearly, $S^I \subseteq S_I$, since separators are upwards closed. Again, the second separator $S^I$ will be the relevant one as it used in the following result:

\begin{theo}{strongarrowtripos} {\rm (Strong arrow tripos)}
Let $A = (A, \cleq, \to, S)$ be a strong arrow algebra. If to each set $I$ we assign the preHeyting algebra $PI = (A^I, \vdash_{S^I})$, while to each function $f: I \to J$ we assign the mapping $PJ \to PI$ sending a function $\varphi: J \to A$ to the composition $\varphi \circ f: I \to A$, then $P$ is a functor $P: {\bf Set}^{\rm op} \to {\bf preHeyt}$. Indeed, $P$ is a tripos whose generic element is the identity function $1_A: A \to A \in PA$.
\end{theo}

\begin{proof} This one in turn follows from \cite[Theorem 5.8]{ferrersantosetal17} and \refcoro{strongalgebrasareimplicativeocas}. We will again outline a proof here.

For $\varphi, \psi: I \to A$ we will write $\varphi \vdash_I \psi$ as an abbreviation of $\varphi \vdash_{S^I} \psi$. Note that this is the case precisely when there is an $r \in S$ such that $r\varphi(i) \cleq \psi(i)$ holds for all $i \in I$. Instead of $\varphi(i)$ we will often write $\varphi_i$ to reduce the use of brackets.

Because $f^* = Pf$ preserves the implication and infima for any $f: I \to J$, it will preserve the interpretation of $\lambda$-terms and hence the preHeyting algebra structure. The functoriality of $P$ is clear, so we do indeed obtain a functor $P: {\bf Set}^{\rm op} \to {\bf preHeyt}$.

We move on to the left and right adjoints. Let $f: I \to J$ be a function and define for $j \in J$:
\begin{eqnarray*}
\forall_f(\alpha)_j & := & \bigcurlywedge_{i \in f^{-1}(j)}\alpha_i \\
\exists_f(\alpha)_j & := & \bigcurlywedge_c\big( \, \bigcurlywedge_{i \in f^{-1}(j)}\alpha_i \to c \, \big)\to c
\end{eqnarray*}
We can prove $$f^*\beta\vdash_J \alpha \Longleftrightarrow\beta\vdash_I\forall_f\alpha$$ by showing:
\begin{displaymath}
    \begin{array}{lclc}
    f^*\beta\to\alpha\in S^I &\Longleftrightarrow &\bigcurlywedge_{i\in I} \beta_{f(i)}\to \alpha_i\in S &\Longleftrightarrow\\ \\
    \bigcurlywedge_{j\in J}\bigcurlywedge_{i \in f^{-1}(j)}\beta_j\to \alpha_i\in S&\Longleftrightarrow&\bigcurlywedge_{j\in J}\beta_j\to\bigcurlywedge_{i \in f^{-1}(j)} \alpha_{i}\in S &\Longleftrightarrow\\ \\ 
    \beta\to\forall_f\alpha\in S^J
    \end{array}
\end{displaymath}
    Here the third equivalence follows from $\mathbf{a}'\in S$, using that $A$ is a strong arrow algebra.

Next, we would like to show $$\alpha\vdash_I f^*\beta \Longleftrightarrow \exists_f\alpha\vdash_J\beta.$$
Let us first suppose that $r \in S$ is such that $r \cleq \alpha_i \to \beta_{fi}$ holds for all $i \in I$. If $j \in J$, then $r \cleq \bcw_{i \in f^{-1}(j)} \alpha_i \to \beta_j$ and hence for $s := (\lambda x.xr)^A \in S$ we have
\begin{eqnarray*}
    s(\exists_f(\alpha))_j & \cleq & \exists_f(\alpha)_jr \\
    & \cleq & \big( \, (\bcw_{i \in f^{-1}(j)} \alpha_i \to \beta_j) \to \beta_j \, \big) r \\
    & \cleq & \beta_j.
\end{eqnarray*}

For the other direction, suppose $\exists_f\alpha\vdash_J\beta$, so $s:=\bigcurlywedge_{j\in J}\exists_f\alpha_j\to \beta_j\in S$. We calculate:\begin{eqnarray*}
   (\lambda x.s(\lambda y.yx))^A
   &= & \bigcurlywedge_{x}x\to s(\bigcurlywedge_y y\to yx)\\
   &\preccurlyeq & \bigcurlywedge_{j\in J}\bigcurlywedge_{i \in f^{-1}(j)}\alpha_i\to s\big(\bigcurlywedge_{c}\big(\bigcurlywedge_{i' \in f^{-1}(j)}\alpha_{i'}\to c\big)\to \big(\bigcurlywedge_{i' \in f^{-1}(j)}\alpha_{i'}\to c\big)\alpha_i\big)\\
   &\preccurlyeq & \bigcurlywedge_{j\in J}\bigcurlywedge_{i \in f^{-1}(j)}\alpha_i\to s\big(\bigcurlywedge_{c}\big(\bigcurlywedge_{i' \in f^{-1}(j)}\alpha_{i'}\to c\big)\to c)\\
   &= & \bigcurlywedge_{j\in J}\bigcurlywedge_{i \in f^{-1}(j)}\alpha_i\to \big(\bigcurlywedge_{j\in J}\exists_f\alpha_j\to \beta_j\big)\exists_f\alpha_j\\
   &\preccurlyeq & \bigcurlywedge_{j\in J}\bigcurlywedge_{i \in f^{-1}(j)}\alpha_i\to \beta_j\\
   &= & \bigcurlywedge_{i\in I}\alpha_i\to \beta_{f(i)} \in S.
\end{eqnarray*}
Hence $\alpha\vdash_I f^*\beta$, as desired.

Now for the Beck-Chevalley condition: it suffices to consider those pullback squares in $\mathbf{Set}$
\[\begin{tikzcd}
	I && {I_1} \\
	\\
	{I_2} && J
	\arrow["{f_1}", from=1-1, to=1-3]
	\arrow["{f_2}"', from=1-1, to=3-1]
	\arrow["{g_2}"', from=3-1, to=3-3]
	\arrow["{g_1}", from=1-3, to=3-3]
\end{tikzcd}\]
for which \begin{itemize}
    \item $I=\{(a,b)\in I_1\times I_2\mid g_1(a)=g_2(b)\}$
    \item the $f_i$ are just the projections to the $i$th coordinate.
\end{itemize}

We need to check that the square \[\begin{tikzcd}
	PI && {PI_1} \\
	\\
	{PI_2} && PJ
	\arrow["{Pf_2}", from=3-1, to=1-1]
	\arrow["{\forall_{f_1}}", from=1-1, to=1-3]
	\arrow["{\forall_{g_2}}"', from=3-1, to=3-3]
	\arrow["{Pg_1}"', from=3-3, to=1-3]
\end{tikzcd}\] commutes. So let $\alpha\in PI_2$. Then \begin{displaymath} \begin{array}{lclc}
    Pg_1\circ\forall_{g_2}(\alpha)&=&\big((\forall_{g_2}\alpha)_{g_1(i_1)}\big)_{i_1\in I_1}
    &= \\ \\ \Big(\bigcurlywedge_{\stackrel{i_2
    \in I_2}{g_2(i_2)=g_1(i_1)}} \alpha_{i_2}\Big)_{i_1\in I_1}&=&\Big(\bigcurlywedge_{\stackrel{(i_1,i_2)
    \in I}{f_1(i_1,i_2)=i_1'}} \alpha_{i_2}\Big)_{i_1'\in I_1}
    &= \\ \\ \forall_{f_1}(\alpha_{i_2})_{(i_1,i_2)\in I}&= & \forall_{f_1}\circ P{f_2}(\alpha)
\end{array}
\end{displaymath}

Finally, for the generic element we take $\Sigma:=A$ and $\sigma:= 1_A\in PA$.
\end{proof}

If the arrow algebra is compatible with joins, we can give a much simpler description of the existential quantifier, as follows.

\begin{lemm}{CompExt}
    If the arrow algebra $A$ is compatible with joins we can define the left adjoint by $$\exists_f(\alpha):= \big(\bigcurlyvee_{f(i)=j}\alpha_i\big)_{j\in J}. $$
\end{lemm}
\begin{proof}
    This follows from the following straightforward calculation: \begin{eqnarray*}\bigcurlywedge_{i\in I}\alpha_i\to (f^*\beta)_i&= & 
        \bigcurlywedge_{i\in I} \alpha_i\to\beta_{f(i)}\\&= & \bigcurlywedge_{j\in J}\bigcurlywedge_{f(i)=j}(\alpha_i\to\beta_{j})\\
        &= & \bigcurlywedge_{j\in J}(\bigcurlyvee_{f(i)=j}\alpha_i)\to\beta_j \\ & = & \bigcurlywedge_{j\in J}(\exists_f\alpha)_j\to\beta_j.
    \end{eqnarray*}
\end{proof}

\subsection{Triposes from general arrow algebras}
We will now show that every arrow algebra gives rise to a tripos. We will do this by showing that every arrow algebra is equivalent to a strong one, in that they give rise to equivalent indexed preorders.

\begin{defi}{functional} Let $(A, \cleq)$ be an arrow structure. We say that an element $a \in A$ is \emph{functional} if
    \[ a = \bcw_{b,c \in A} \{ b \to c \, : \, a \cleq b \to c \}. \]
We write $A_{\rm fun}$ for the collection of functional elements in $A$. 
\end{defi}

The following is immediate from the definitions:

\begin{lemm}{propertiesofregularelements}
    Implications are functional and functional elements are closed under arbitrary infima $\bcw$. Hence the combinators $\km, \sm, \am$ are functional.
\end{lemm}

\begin{coro}{anotherarrowstructure}
    If $(A, \cleq, \to)$ is an arrow structure, then so is $(A_{\rm fun}, \cleq, \to)$. Indeed, if $(A, \cleq, \to, S)$ is an arrow algebra, then so is $(A_ {\rm fun}, \cleq, \to, S_{\rm fun})$ where $S_{\rm fun} = S \cap A_{\rm fun}$.
\end{coro}

\begin{prop}{combaforregularelements}
    Let $A$ be an arrow structure, $a \in A$ and $B \subseteq A_{\rm fun}$. Then
    \[ \am \cleq (\bcw_{b \in B} a \to b) \to a \to \bcw_{b \in B} b. \]
    Therefore the arrow algebra $(A_ {\rm fun}, \cleq, \to, S_{\rm fun})$ will be strong for any arrow algebra $(A, \cleq, \to, S)$.
\end{prop}
\begin{proof}
    For $b \in B$, let us write
    \[ C_b := \{ x \to y \, : \, b \cleq x \to y \} \]
    and $C := \bigcup C_b$, so that $b = \bcw C_b$ and $\bcw_{c \in C} c = \bcw_{b \in B} b$. If $c \in C$, then $c \in C_{b_0}$ for some $b_0 \in B$, and therefore
    \[ \bcw_{b \in B} a \to b \cleq a \to b_0 \cleq a \to c. \]
    This implies
    \[ \bcw_{b \in B} a \to b \cleq \bcw_{c \in C} a \to c, \]
    from which it follows that
    \begin{eqnarray*}
        \am & \cleq & (\bcw_{c \in C} a \to c) \to a \to \bcw_{c \in C} c \\
        & \cleq & (\bcw_{b \in B} a \to b) \to a \to \bcw_{b \in B} b
    \end{eqnarray*}
\end{proof}

The following result shows that the arrow algebras $A_{\rm fun}$ and $A$ are ``equivalent''.

\begin{prop}{partialcanberemoved} Let $A$ be an arrow algebra with separator $S$. Every element $a \in A$ is logically equivalent to the functional element $\top \to a$, in a uniform fashion. More precisely, we have
    \[ \bigcurlywedge_{a \in A} a \to (\top \to a) \in S \quad \mbox{ and } \bigcurlywedge_{a \in A} (\top \to a) \to a \in S. \]
It follows that the preHeyting algebras and indexed preorders induced by $A$ and $A_{\rm fun}$ are equivalent.
\end{prop}
\begin{proof}
    The former follows from ${\bf k} \in S$. To prove the latter, we note that $B \to (B \to A) \to A$ is intuitionistic tautology, so
    \[ \bigcurlywedge_{a,b} b \to (b \to a) \to a \in S \]
    and
    \[ \bigcurlywedge_{a} \top \to (\top \to a) \to a \in S. \]
    Since $\top \in S$, the result follows from \reflemm{sepclosedunderlimitedMP}.
\end{proof}

In particular we can deduce:

\begin{prop}{Heytconstr} 
    Let $A=(A,\preccurlyeq,\to,S)$ be an arrow algebra. In that case we can define a preorder $\vdash$ on $A$ as follows: 
    $$a\vdash b\text{ if }a\to b\in S.$$
    We refer to $\vdash$ as the \emph{logical ordering} and with this preorder, $(A,\vdash)$ is a preHeyting algebra with $\to$ as Heyting implication.
\end{prop}

\begin{theo}{arrowtripos} {\rm (Arrow tripos)}
    Let $A = (A, \cleq, \to, S)$ be an arrow algebra. If to each set $I$ we assign the preHeyting algebra $PI = (A^I, \vdash_{S^I})$, while to each function $f: I \to J$ we assign the mapping $PJ \to PI$ sending a function $\varphi: J \to A$ to the composition $\varphi \circ f: I \to A$, then $P$ is a functor $P: {\bf Set}^{\rm op} \to {\bf preHeyt}$. Indeed, $P$ is a tripos whose generic element is the identity function $1_A: A \to A \in PA$.
    \end{theo}
    
\begin{rema}{workofumberto} In subsequent work by Umberto Tarantino \cite{tarantino25} the reader can find a precise definition of equivalence of arrow algebras, which implies that the associated indexed preorders are equivalent. \refprop{partialcanberemoved} implies that the arrow algebras $A$ and $A_{\rm fun}$ are equivalent in Tarantino's sense.
\end{rema}

\begin{rema}{equivalentimplicativealgebra}
    It follows from Miquel's work \cite{miquel20b} that every arrow algebra is equivalent to an implicative algebra. The construction is as follows: if $A$ is an arrow algebra, then write $I_0 := {\rm Pow}(A)^+ \times A$ and $I = {\rm Pow}(I_0)$. We have a map $\varphi_0: I_0 \to A$ defined as
    \[ \varphi_0(Y_1,\ldots,Y_n,y) := (\bcw Y_1) \to (\bcw Y_2) \to \ldots \to (\bcw Y_n) \to y, \]
    where $Y_i \subseteq A, y \in A$. In particular, we will write $\varphi_0[i] = \{ \varphi_0(s) \, : \, s \in i \} \subseteq A$ for any $i \in I$. We can now give $I$ an implicative structure by ordering it through reverse inclusion and writing
    \[ i \to j := \{ (X,Y_1,\ldots,Y_n,y) \, : \, \varphi_0[i] \subseteq X, (Y_1,\ldots,Y_n,y) \in j \} \]
    for any $i, j \in I$. If $S^A$ is the separator on $A$, then
    \[ S^I := \{ i \in I \, : \, \bcw \varphi_0[i] \in S_A \} \]
    defines a separator on $I$, with the implicative algebra $(I, S^I)$ inducing a tripos equivalent to the one induced by $(A, S^A)$. (Indeed, the map $\varphi: I \to A$ defined by $\varphi(i) = \bcw \phi_0[i]$ is an equivalence in Tarantino's sense.) Since implicative algebras are in particular strong arrow algebras, this provides a different proof of the fact that every arrow is equivalent to a strong one. However, due to the complicated nature of this construction, it is much easier to understand the tripos associated to $(A, S^A)$ through $(A_{\rm fun}, S^A_{\rm fun})$ rather than through $(I, S^I)$. Indeed, we do not think that trying to understand the tripos induced by $(A, S^A)$ through the implicative algebra $(I, S^I)$ is going to be very illuminating.
\end{rema}    

\section{Subtriposes from nuclei}

We have seen in \refprop{subarrowalgebras} that whenever $j$ is a nucleus on an arrow algebra $A$, we obtain a new arrow algebra $A_j$. The result in the previous section says that we obtain triposes from both $A_j$ and $A$. The natural question is how these triposes are related. Locale theory suggests that the tripos associated to $A_j$ should be a subtripos of the one associated to $A$ and in this section we will show that that is indeed the case.

So let us first define what we mean by a subtripos.

\begin{defi}{geometricmorphism} (Geometric morphism) Let $P$ and $Q$ be triposes over $\mathbf{Set}$. A \emph{geometric morphism} $Q\to P$ is a pair $\Phi=(\Phi_+,\Phi^+)$ consisting of natural transformations $\Phi^+:P\to Q$ and $\Phi_+:Q\to P$. We require that for any set $X$ the monotone map $\Phi^+_X: P(X)\to Q(X)$ is left adjoint to $\Phi_+^X$ and preserves finite meets.
\end{defi}

\begin{defi}{inclusionoftriposes} (Inclusion of triposes) Let $P$ and $Q$ be triposes over $\mathbf{Set}$. A geometric morphism $\Phi=(\Phi_+,\Phi^+):Q\to P$ is called an \emph{inclusion of triposes} if $\Phi_+^X$ is full for every set $X$. (In other words, the monotone map $\Phi_+^X$ also reflects the preorder.) In this case we call $Q$ a \emph{subtripos} of $P$.
\end{defi}

These definitions allow us to state the result we wish to prove.

\begin{prop}{subtriposfromnucleus} {\rm (Subtripos induced by a nucleus)}
    Let $A$ be an arrow algebra and $j$ a nucleus on $A$. Let $P$ be the tripos induced by $A$ and $P_j$ the one induced by $A_j$. Then $P_j$ is a subtripos of $P$.
\end{prop}

To prove it, we need the following lemma, which is a continuation of \reflemm{nuclproperties}.

\begin{lemm}{morenuclproperties} Suppose $A$ is an arrow algebra with separator $S$. If $j$ is a nucleus on $A$ and $\wedge$ the meet in the logical ordering as in \refprop{Heytconstr}, then we have:
    \begin{enumerate}
        \item[(iv)] $\bigcurlywedge_{a,b\in A} j(a \land b)\rightarrow(ja \land jb)\in S$ and $\bigcurlywedge_{a,b\in A}(ja\land jb)\to j(a\land b)\in S$.
        \item[(v)] $j\big(\bigcurlywedge_{x\in X}x\big)\preccurlyeq\bigcurlywedge_{x\in X}jx$ for any $X\subseteq A$.
    \end{enumerate}
\end{lemm}
\begin{proof}
    \begin{enumerate}
        \item[(iv)] This follows from $\bigcurlywedge_{a,b\in A}j(a\to b)\to (ja\to jb)\in S$ (point (iii) of \reflemm{nuclproperties}) combined with the knowledge that $\land$ and $\to$ act as the conjunction and implication.

        Indeed, from $\bigcurlywedge_{a,b}a\land b\to a\in S$ and $\bigcurlywedge_{a,b}a\land b\to b\in S$, we get $\bigcurlywedge_{a,b}j(a\land b)\to ja\in S$ and $\bigcurlywedge_{a,b}j(a\land b)\to jb\in S$. This gives the first claim.

        For the other, we start by noting $\bigcurlywedge_{a,b}a\to (b\to a\land b)\in S$, and thus by using point (iii) twice, we obtain $\bigcurlywedge_{a,b}ja\to (jb\to j(a\land b))\in S$.
        \item[(v)] For all $x\in X$ we have $\bigcurlywedge_{x\in X}x\preccurlyeq x$. By monotonicity we have $j\big(\bigcurlywedge_{x\in X}x\big)\preccurlyeq jx$ for all $x\in X$, giving the desired result.
    \end{enumerate}
\end{proof}

\begin{proof} (Of \refprop{subtriposfromnucleus}.) Before we start, it should be noted that if $j$ is a nucleus on $A$ with separator $S$, then $j^X = j \circ -$ is also a nucleus on $A^X$ with separator $S^X$. This means that \reflemm{nuclproperties} and \reflemm{morenuclproperties} can be applied to $j^X$ as well.

Moreover, we have \[ \alpha \vdash_j \beta \mbox{ in } P_j(X) \mbox{ if and only if } \alpha \to j^X \beta \in S^X. \] In other words, we claim that $$j\big(\bigcurlywedge_{x\in X}\alpha(x)\to j\beta(x)\big)\in S\Longleftrightarrow\bigcurlywedge_{x\in X}\alpha(x)\to j\beta(x)\in S.$$ The right to left implication is trivial. For the left to right direction we first use point (v) of \reflemm{morenuclproperties} to see that $j\big(\bigcurlywedge_{x\in X}\alpha(x)\to j\beta(x)\big)\in S$ implies $j^X(\alpha\to j^X\beta)\in S^X$. From here we use \reflemm{nuclproperties} for $j^X$: 
\begin{align*}
    &j^X(\alpha\to j^X\beta)\in S^X&&\Rightarrow&& j^X\alpha\to j^Xj^X\beta\in S^X&&\Rightarrow\\
    &j^X\alpha\to j^X\beta\in S^X&&\Rightarrow&&\alpha\to j^X\beta\in S^X
\end{align*}

From this and the fact that $j^X$ is a nucleus it follows fairly easily that $\Phi_+$ and $\Phi^+$ defined as:\begin{align*}
            &\Phi_{+}^X:P_j X\to PX:\alpha\mapsto j\circ\alpha\\
            &\Phi_X^+:PX\to P_j X:\alpha\mapsto\alpha
        \end{align*} are natural transformations, which are pointwise adjoints, in the sense that $$\Phi^+_X\alpha\vdash_j\beta\Longleftrightarrow\alpha\vdash\Phi^X_+\beta$$ for any set $X$. In addition, it also follows easily from this that $\Phi_+$ at each point preserves and reflects the preorder, that is, $\alpha\vdash_j\beta\Leftrightarrow j\alpha\vdash j\beta$.

        What remains to be checked is that $\Phi^+$ preserves finite meets. However, if $\gamma \vdash_j \alpha$ and $\gamma \vdash_j \beta$, then $\gamma \to j^X\alpha \in S^X$ and $\gamma \to j^X\beta \in S^X$. Therefore $\gamma \to j^X\alpha \land j^X\beta \in S^X$ and $\gamma \to j^X(\alpha \land \beta) \in S^X$ by (iv) of \reflemm{morenuclproperties}. Hence $\gamma \vdash_j \alpha \land \beta$, as desired.
    \end{proof}

\begin{rema}{everysubtriposfromnucleus} So every nucleus on an arrow algebra gives rise to a subtripos of the tripos associated to the arrow algebra. In \cite{tarantino25} Tarantino has shown that every subtripos arises in that way.
\end{rema}

The tripos $P_j$ we have just constructed is equivalent to another tripos which we will denote $P^j$. It turns out that the latter tripos is often more useful for doing calculations. Indeed, the main benefit of $P^j$ is that it is, pointwise, a subpreorder of $P$. That is, the preorder is the same, it only has fewer elements.

\begin{prop}{nucltripagain}
    The tripos $P_j$ and functor $$P^j: I\mapsto\{\alpha\in PI\mid j^I\alpha\vdash\alpha\}$$ are equivalent, making $P^j$ also a tripos.
    \begin{proof}
        For this we use the geometric morphism as defined in the proof of \refprop{subtriposfromnucleus}. It is easy to see that $\Phi:P_j\to P^j$ is well-defined. What is left is to show is that $\Phi^+\Phi_+$ and $\Phi_+\Phi^+$ are equivalent to the identities. 
        
        For this let $I$ be a set and $\alpha\in P_j I$. Now $\Phi^+_I\Phi_+^I\alpha=j^I\alpha$. From the proof of the previous proposition it easily follows that we have $j^I\alpha\vdash_j\alpha$. From here we can conclude that $j^I\alpha$ and $\alpha$ are equivalent because $\alpha \vdash j^I \alpha$ and $\alpha\vdash_j j^I\alpha$ are true in virtue of $j^I$ being a nucleus.

        Now let $\alpha\in P^jI$. Then $\Phi_+^I\Phi_I^+\alpha=j^I\alpha$, but this is equivalent to $\alpha$ by definition of $P^j$.
    \end{proof}
\end{prop}

The following corollary will prove useful in the next section:

\begin{coro}{nucltripagainplus} Suppose $j$ is a nucleus on an arrow algebra $A$ with the property that $j^2 = j$ as maps $A \to A$. Then the tripos $P_j$ and functor $$\overline{P}^j: I\mapsto\{\alpha\in PI\mid j^I\alpha = \alpha\}$$ are equivalent, making $\overline{P}^j$ also a tripos.
    \begin{proof}
        If $j^2 = j$ then this tripos is equivalent to the one from the previous proposition.
    \end{proof}
\end{coro}

\begin{rema}{connectivesinPj} For future reference, it will be worthwhile to make explicit how the tripos structure in $P^j$ is calculated.
    \begin{enumerate}
        \item Implication and meets are calculated as in $P$; this follows from items (iii) and (iv) of \reflemm{nuclproperties} and \reflemm{morenuclproperties} respectively.
        
        \item The joins of $P^j$ are those of $P$ with $j$ applied to them. To see this, suppose $\alpha,\beta,\gamma\in P^jI$. Note that $\alpha\vdash\alpha \lor \beta$ implies $j\alpha\vdash j(\alpha \lor \beta)$, which in combination with $\alpha \in P^j$ gives $\alpha\vdash j(\alpha \lor \beta)$. (Similarly for $\beta$.)

        Now suppose $\alpha\vdash\gamma$ and $\beta\vdash\gamma$. Then $\alpha \lor \beta\vdash\gamma$ implies $j(\alpha \lor \beta)\vdash j\gamma$, which in combination with $\gamma \in P^j$ implies $j(\alpha \lor \beta)\vdash \gamma$.

        \item Now we claim the universal adjoint can stay the same and the existential will be the old one with $j$ applied to it. 
        
        To see this we note that if $\alpha \in P^j$, then:
        \begin{displaymath}
            \begin{array}{lclc}
                f^*\forall_f\alpha\vdash \alpha &\Rightarrow & jf^*\forall_f\alpha\vdash j\alpha & \Rightarrow\\
                f^*j\forall_f\alpha\vdash j\alpha&\Rightarrow&
                f^* j\forall_f\alpha\vdash \alpha & \Rightarrow\\
                j\forall_f\alpha\vdash\forall_f\alpha
            \end{array}
        \end{displaymath}

        In addition, we have
        \[ j\exists_f\alpha\vdash\beta \Rightarrow \exists_f \alpha \vdash \beta \Rightarrow \alpha \vdash f^* \beta \]
        and
        \[ \alpha \vdash f^* \beta \Rightarrow \exists_f \alpha \vdash \beta \Rightarrow j\exists_f \alpha \vdash j\beta \Rightarrow j\exists_f \alpha \vdash \beta, \]
        if $\alpha, \beta \in P^j$.
        \item As a generic element we can take $j: A \to A \in PA$.
    \end{enumerate}
\end{rema}

\begin{rema}{twophilosophies} The two equivalent approaches to the subtriposes induced by a nucleus correspond to the two different ways of approaching the negative translation (see \cite{vandenberg19}). The first one, which is we denoted with $P_j$, is most closely associated to the Kuroda negative translation. On this approach we keep the usual predicates and change the notion of consequence. The second one, which is the one we denoted with $P^j$, is to keep the usual notion of consequence, but restrict the predicates: this is most closely associated to the G\"odel-Gentzen negative translation. The second view is often more pleasant to work with because implication and universal quantification remain the same.
\end{rema}

\section{Applications}

In this section we wish to illustrate the flexibility of the framework we have created by showing how various modified realizability toposes fit within it. Indeed, the power is that we can construct triposes without too much effort (and without going through the laborious task of checking that they define triposes). We will not further investigate the resulting toposes -- that is a task that we hope to be able to take up in future work. 

\subsection{The modification of an arrow algebra} Let us start by showing that arrow algebras are closed under another construction provided they are \emph{binary implicative}.

\begin{defi}{binaryimplicative} Let $A$ be an arrow structure. We will call $A$ \emph{binary implicative} if
    $$a\to(b\curlywedge b')=(a\to b)\curlywedge( a\to b')$$
for all $a, b, b' \in A$. An arrow algebra will be called binary implicative if its underlying arrow structure is.
\end{defi}
Clearly, implicative algebras are also binary implicative.

\begin{prop}{Sierpinskiconstruction} {\rm (Sierpi\'nski construction)}
    Let $A$ be an binary implicative arrow algebra. Define $$A^\to:=\{x = (x_0,x_1)\in A^2\mid x_0 \preccurlyeq x_1\}$$

    If we order $A^\to$ pointwise, and define implication as $$x\to y:=(x_0\to y_0 \curlywedge x_1\to y_1,x_1\to y_1)$$ then $A^\to$ carries an arrow structure. Moreover,  $$S^\to:=\{x\in A^\to\mid x_0\in S\}$$ defines a separator on this arrow structure. 
\end{prop}
\begin{proof}
    Since the proof is lengthy and tedious, we have relegated it to an appendix.
\end{proof}

The following lemma follows directly from intuitionistic logic.
\begin{lemm}{disjnucl}
    Let $A$ be an arrow algebra and $a \in A$. Then \[ j: A \to A: x \mapsto x \lor a \] defines a nucleus, where $\lor$ refers to the join in the logical ordering (see \refprop{Heytconstr}).
\end{lemm}

\begin{defi}{modificationarrowalg}
    Let $A$ be a binary implicative arrow algebra and $j$ be the nucleus on $A^\to$ defined by $j(x) = x \lor (\bot, \top)$. Then we will refer to
    \[ A^m := A^\to_j \]
    as the \emph{modification of the arrow algebra $A$}.
\end{defi}    

\subsection{Modified realizability} In its original incarnation, modified realizability was a proof-theoretic interpretation. It was invented by Kreisel in 1959, who used it to show the independence of Markov's Principle over intuitionistic arithmetic in finite types \cite{kreisel59}. Later, based on ideas by Hyland and Troelstra, Grayson gave a semantic counterpart to Kreisel's interpretation and defined the modified realizability tripos \cite{grayson81}. (Other references for modified realisability toposes include \cite{hylandong93,vanoosten97,birkedalvanoosten02,johnstone17}.) His definition was as follows; here we use that $\mathbb{N}$ carries the structure of an absolute and discrete pca ($K_1$, see \refexam{examplepca}).

\begin{defi}{graysonstopos} (Grayson's modified realizability tripos) Write
    $$\Sigma_M:=\{(A_0,A_1)\mid A_0\subseteq A_1\subseteq\mathbb{N}, 0\in A_1\}.$$
    Then we can define a functor $$M:\mathbf{Set}^{op}\to\mathbf{Preorder}$$ which sends a set $I$ to the set of functions $I \to \Sigma_M$ preordered as follows: \begin{align*}
        \alpha\leq_{MI} \beta &&\text{ if }&&\exists r\in\mathbb{N} \, \forall i\in I\, \big(  \, \forall n\in\alpha_1(i) \,  \, rn\downarrow\text{ and }rn\in\beta_1(i) \, \big)\\&&&&\text{ and } \big( \, \forall n\in\alpha_0(i) \, rn\in\beta_0(i) \,  \, \big)
    \end{align*}
    In addition, the functor $M$ acts on morphisms by precomposition.
\end{defi}

The following proposition explains how this indexed preorder is related to the modification of an arrow algebra. Indeed, recall that any pca induces an arrow algebra (as in \reftheo{arrowalgfrompcas}), so that we have an arrow algebra $D(K_1)$. (In fact, this arrow algebra is binary implicative.) The tripos induced by its modification is equivalent to $M$.

\begin{prop}{MRTequiv}
    As an indexed preorder, Grayson's modified realizability tripos is equivalent to the tripos induced by the modification of the arrow algebra $D(K_1)$. Hence Grayson's modified realizability tripos is indeed a tripos.
    \begin{proof}
        Let $A = D(K_1)$ be the arrow algebra induced by $K_1$, while $P$ is the tripos induced by $A^\to$ and $P^j$ the subtripos of $P$ induced by $jx = x \lor (\bot, \top)$, as explained in \refprop{nucltripagain}. 

        There is an obvious natural transformation $F_I: MI \to PI$ induced by the inclusion $\Sigma_M \to D(K_1)^\to$. Indeed, its pointwise full and faithful, because 
        \begin{align*}
            \alpha\leq_{MI}\beta &&\Leftrightarrow&&\bigcurlywedge_{i\in I}\alpha(i)\to\beta(i)\in S^\to&&\Leftrightarrow&& \alpha\vdash_{PI}\beta
        \end{align*}
        So it remains to show that the essential image of $F$ is precisely $P^j$.

        For that purpose let us determine what it means for an element $\alpha \in PI$ to belong to $P^jI$. In order for $\alpha\in P^jI$ to hold, we must have $$\bigcurlywedge_{i\in I} (\alpha(i) \lor(\emptyset,\mathbb{N}))\to\alpha(i)\in S^\to,$$ which 
        by intuitionistic reasoning is equivalent to $$\bigcurlywedge_{i\in I}(\emptyset,\mathbb{N})\to\alpha(i)\in S^\to.$$ Using the definition of $S^\to$, this is equivalent to$$\bigcap_{i\in I}\big((\emptyset,\mathbb{N})\to\alpha(i)\big)_0 \mbox{ is inhabited.}$$
        Because $((\emptyset,\mathbb{N})\to\alpha(i))_0=\emptyset\to\alpha_0(i)\cap\mathbb{N}\to\alpha_1(i)=\mathbb{N}\to\alpha_1(i)$, this is equivalent to $$\bigcap_{i\in I}\mathbb{N}\to\alpha_1(i) \mbox{ is inhabited,}$$
        which in turn is equivalent to \[ \bigcap_{i\in I}\alpha_1(i) \mbox{ is inhabited.} \] 
        
        Since $0 \in \bigcap_{i \in I }\alpha_1(i)$ for any element $\alpha \in M$, every element in $M$ lands in $P^j$ under $F$. Conversely, suppose $\alpha\in P^jI$ and  let $a\in\bigcap_{i\in I}\alpha_1(i)$. Now define $$f:\mathbb{N}\to\mathbb{N}:x\mapsto\begin{cases}0&\mbox{if } x=a\\
        a&\mbox{if }x=0\\
        x&\text{else}\end{cases}$$

        By postcomposition with $f$ (on both components) we get $f\circ\alpha\in MI$. Because $f$ is a recursive bijection we can code it as a natural number $e$. With this we get $$e\in\bigcurlywedge_{i\in I}\alpha(i)\to f\circ\alpha(i)$$
        which implies $\alpha\vdash_{PI}f\circ\alpha$. Using the code of the inverse of $f$, we can prove $f\circ\alpha\vdash_{PI}\alpha$. We conclude that $F: M \to P^j$ is pointwise essentially surjective.
    \end{proof}
\end{prop}

\begin{rema}{equivvsequiv}
    Note that we have constructed a natural transformation $F: M \to P^j$ which is pointwise fully faithful and essentially surjective. In a constructive metatheory this seems to be the best possible: we do not see how to prove the existence of a pseudo-inverse without using the Axiom of Choice.
\end{rema}

\begin{rema}{comparisontojvoostenbirkedal}
    This analysis of modified realizability is a variation on the observation in \cite{vanoosten97} that the modified realizability topos is a closed subtopos of the effective topos over Sierpi\'nski space.
\end{rema}

\subsection{Extensional modified realizability topos} The modification construction can also be applied to ${\rm PER}(K_1)$, the arrow algebra of PERs on the natural numbers, which is the result of applying the construction from Subsection 3.2.3 on the absolute and discrete pca $K_1$. Again, this arrow algebra is binary implicative.

\begin{defi}{triposT} Let $$\Sigma_T:=\{(R,S)\in\text{PER}(K_1)\mid R\subseteq S,(0,0)\in S\}.$$ Then we can define a functor
\[ T: {\bf Sets}^{\rm op} \to {\bf Preorder} \] by putting $TI = \Sigma_T^I$ and
    \begin{align*}
        \alpha\leq_{TI}\beta&&\text{if} && \exists r\in\mathbb{N} \, \forall i\in I \, \big( \, r\in \alpha_0(i)\to\beta_0(i)\cap \alpha_1(i)\to\beta_1(i) \, \big)
    \end{align*}
    for any  $\alpha,\beta\in\Sigma^X$. Moreover, $T$ acts on morphisms by precomposition.
\end{defi}

\begin{prop}{PERtrip}
    The indexed preorder $T$ and the tripos induced by the modification of ${\rm PER}(K_1)$ are equivalent as indexed preorders. Hence $T$ is also a tripos.
    \begin{proof}
        The proof is similar to that of \refprop{MRTequiv}. Write $A = {\rm PER}(K_1)$ and denote the tripos induced by $A^\to$ by $P$ and its subtripos induced by the nucleus $j(x) = x \lor (\bot, \top)$, computed as in \refprop{nucltripagain}, by $P^j$. There is again an obvious natural transformation $F: T \to P$ induced by the inclusion $\Sigma_T \to {\rm PER}(K_1)^\to$, which is pointwise full and faithful. To show that its essential image is $P^j$, first note that in this case $(\bot,\top)=(\emptyset,\N\times\N)$. Then an element $\alpha \in PI$ belongs to $P^jI$ precisely when $$\bigcap_{i\in I}\N\times\N\to\alpha_1(i)\mbox{ is inhabited.}$$
        This shows the image of $F$ is included in $P^j$. Conversely, if $(a,b)\in \bigcap_{i\in I}\alpha_1(i)$, then also $$(a,a)\in\bigcap_{i\in I}\alpha_1(i),$$ because   $\alpha_1(i)$ is symmetric and transitive. This means that if $f$ is the recursive bijection from \refprop{MRTequiv}, then for $$\tilde{f}:\N^2\to\N^2:(x,y)\mapsto (fx,fy)$$ we have $\tilde{f}\circ\alpha\in TI$ as well as both $\alpha\vdash\tilde{f}\circ\alpha$ and $\tilde{f}\circ\alpha \vdash \alpha$.
    \end{proof}
\end{prop}

In his MSc thesis \cite{devries17} (supervised by the first author), Mees de Vries defines a subtripos of $T$. One of his results is that in the internal logic of the resulting topos both the Axiom of Choice for all finite types as well as the Independence of Premise principle hold. This can be seen as an improvement over Grayson's topos: indeed, these are characteristic principles of modified realizability, but the Axiom of Choice for finite types fails in Grayson's topos.

The tripos studied by De Vries is constructed as follows.
\begin{defi}{meestripos} Write $$\Sigma_V:=\{(A_0,A_1,\sim)\mid A_0\subseteq A_1\subseteq\mathbb{N},0\in A_1,\sim\text{an equivalence relation on }A_1\}.$$ Then we can define functor
    \[ V: {\bf Sets}^{\rm op} \to {\bf Preorder} \]
by putting $VI = \Sigma^I$ and \begin{align*}
        \alpha\leq_{VI} \beta&&\text{if}&&\exists r\in \mathbb{N} \, \forall i\in I \big( \, \forall n\in\alpha_1(i) \, rn\downarrow\text{ and }rn\in\beta_1(i) \, \big)\\&&&&\text{and } \big( \, \forall n,m\in\alpha_1(i) \, n\sim m\text{ implies }rn\sim rm \, \big)\\&&&&\text{ and } \big( \, \forall n\in\alpha_0(i) \, rn\in\beta_0(x) \, \big)
    \end{align*}
for any $I\in\mathbf{Set}$ and $\alpha,\beta\in\Sigma^I$ (in this connection we refer to $r$ as the \emph{realizer}). Moreover, on morphisms $V$ acts by precomposition.
\end{defi}

To see how the tripos defined defined by De Vries fits into our framework, we first need to make a couple of observations. Let us say that two nuclei $j, k$ on $A$ are equivalent if $\bcw_{a \in A} ja \to ka \in S$ and $\bcw_{a \in A} ka \to ja \in S$ hold.

\begin{lemm}{domnucl} Let us write $B$ for the arrow algebra ${\rm PER}(K_1)^\to$ obtained by applying the Sierpi\'nski construction to the arrow algebra ${\rm PER}(K_1)$. \begin{enumerate}
    \item[(i)] If $j$ and $k$ are nuclei on an arrow algebra $A$ which commute in the sense that
    \[ \bcw_{a \in A} jka \to kja \in S \mbox{ and } \bcw_{a \in A} kja \to jka \in S \]
    hold, then $jk : x \mapsto jkx$ and $kj: x \mapsto kjx$ define equivalent nuclei on $A$.
    \item[(ii)] If $\sigma: \mathbb{N} \mapsto \mathbb{N}$ is the successor map sending $n$ to $n+1$ and we write $\sigma_* R = \{ (\sigma x, \sigma y) \, : (x, y) \in R \}$ for any $R \in {\rm PER}(K_1)$, then
    \[ j: B \to B: (R, S) \mapsto (\sigma_* R, \sigma_* S \cup \{ (0,0 \})) \]
    defines a nucleus on $B$ which is equivalent to the modification nucleus $x \mapsto x \lor (\bot, \top)$.
    \item[(iii)] The operation $$k:B \to B:(R,S)\mapsto(S\upharpoonright{\rm dom} \, R,S)$$ defines a nucleus. Here ${\rm dom}(R) = \{ x \, : \, (x,x) \in R \}$ and $S \upharpoonright X = \{ (x, y) \in S \, : x, y \in X \}$.
    \item[(iv)] The nuclei $j$ and $k$ defined in items (ii) and (iii) commute, so $jk$ and $kj$ define equivalent nuclei.
\end{enumerate}
\begin{proof} The statement in (i) is an exercise is intuitionistic logic which we leave to the reader.
    
For (ii) it suffices to show that $j(R,S)$ has the universal property of $(R, S) \lor (\bot, \top)$. To see this, note that if $(R, S) \to (U, V)$ is realized by $e$ and $(\bot, \top) \to (U, V)$ by $f$, then $j(R, S) \to (U, V)$ is realized by the function which sends $0$ to $f0$ and $n+1$ to $e(n)$. Conversely, if $j(R, S) \to (U, V)$ is realized by $e$, then $(R, S) \to (U, V)$ is realized by the function which sends $n$ to $e(n+1)$, while $(\bot, \top) \to (U, V)$ is realized by the total function which is constant at $e0$.

For (iii) we will only show that $k$ satisfies the the third requirement for being a nucleus as the first two hold trivially. 

    In our case $$\bigcurlywedge_{a,b\in A}(a\to kb)\to(ka\to kb)\in S$$ translates to there being an element $(a,a')$ such that for any $(R,S),(T,U)\in {\rm PER}(\N)^\to$ and any $(x,x')\in ( \, (R\to U\upharpoonright\dom(T) \, )\cap S\to U,S\to U)$, we have $(ax,a'x')\in( \, ( \, S\upharpoonright\dom(R)\to U\upharpoonright\dom(T) \, )\cap S\to U,S\to U)$. We claim this element is just the identity. As you can see this clearly works for the second coordinate, so we only have to worry about the first.

    For this, take $(x,x')\in ( \, R\to U\upharpoonright\dom(T) \, ) \cap S\to U$ and $(s,s')\in S\upharpoonright\dom(R)$. If we get $(xs,x's')\in U\upharpoonright\dom(T)$, we are done. First of all note that because $(x,x')\in S\to U$, we have $(xs,x's')\in U$ and it remains to show that $xs, x's' \in {\rm dom}(T)$. Because $s,s'\in\dom(R)$, we have that $(s,s),(s',s')\in R$ and therefore $(xs,x's),(xs',x's')\in U\upharpoonright\dom(T)$. Hence  $xs, x's' \in {\rm dom}(T)$ and we are done.

Finally, we have $jk = kj$ as functions $B \to B$, so the nuclei $j$ and $k$ certainly commute and therefore the statement in (iv) holds as well.
\end{proof}
\end{lemm}

\begin{prop}{arrowalgfordevries} Let $B^{jk}$ be the arrow algebra obtained from $B = {\rm PER}(\mathbb{N})^\to$ by considering the nucleus $j \circ k$, where $j$ and $k$ are the nuclei from the previous lemma. Then the indexed preorder $V$ is equivalent to the tripos induced by $B^{jk}$. In particular, $V$ is a tripos.
    \begin{proof}
        Let us write $P$ for the tripos induced by $B = {\rm PER}(\mathbb{N})^\to$ and $P^{jk}$ for the subtripos induced by $j \circ k$, computed as in \refprop{nucltripagain}. Note that there is a map $\Sigma_V \to {\rm PER}(\N)^\to$ sending $(A_0, A_1, \sim)$ to $(\sim \upharpoonright A_0, \sim)$, which induces a natural transformation $F: V \to P$ by postcomposition with this map. We can again easily verify that this transformation is pointwise full and faithful, so it remains to show that its pointwise essential image is $P^{jk}$. Note that an element $\alpha \in PI$ belongs to $P^{jk}(I)$ precisely when $j\alpha \vdash \alpha$ and $k\alpha \vdash \alpha$. By an argument similar to that in \refprop{PERtrip} we can verify that any $\alpha$ such that $j\alpha \vdash \alpha$ is equivalent to one with $0 \in \bigcap_{i \in I} \alpha_1(i)$. And then an argument similar to one in \refcoro{nucltripagainplus} shows that any such with $k\alpha \vdash \alpha$ is equivalent to one with $\alpha = k\alpha$, because $k^2 = k$. This shows that any element in $P^{jk}I$ is isomorphic to one in the image of $FI$.
    \end{proof}
\end{prop}

\begin{rema}{ontriposT} We suspect that also in the internal logic of the topos associated to $T$ the characteristic principles of modified realizability hold (Axiom of Choice for all finite types and Independence of Premise). We hope to investigate this further in future work.
\end{rema}    

\section{Conclusion}

We have introduced the notion of an arrow algebra and shown it to be a flexible tool for constructing triposes. This is due to the fact that there are many naturally occurring examples, such as locales and those induced by pcas, and that the notion of an arrow algebra is closed under various constructions, in particular a notion of a subalgebra given by nuclei and a Sierpi\'nski construction. 

The main omission in this paper is that we did not introduce a notion of morphism of arrow algebras, related to the notion of a geometric morphism of toposes. Fortunately, this gap has been filled by Umberto Tarantino \cite{tarantino25}, whose work should be seen as the sequel to this paper.

\bibliographystyle{plain}
\bibliography{hSetoids}

\appendix

\section{Sierpi\'nski construction}

This appendix is devoted to a proof of the following result. We remind the reader that our convention is that $\to$ binds stronger than $\bigcurlywedge$ or $\curlywedge$.

\begin{prop}{Sierp}
    Let $A$ be a binary implicative arrow algebra. Define $$A^\to:=\{x = (x_0,x_1)\in A^2\mid x_0 \preccurlyeq x_1\}$$
    If we order $A^\to$ pointwise, and define implication as $$x\to y:=(x_0\to y_0 \curlywedge x_1\to y_1,x_1\to y_1),$$ then $A^\to$ carries an arrow structure. Moreover,  $$S^\to:=\{x\in A^\to\mid x_0\in S\}$$ defines a separator on this arrow structure.
    \begin{proof}
    It is easy to see $A^\to$ is an arrow structure and that $S^\to$ is upwards closed and closed under modus ponens. The difficulty is showing that $S^\to$ contains the combinators ${\bf k}^\to, {\bf s}^\to$ and ${\bf a}^\to$, which we will do by showing that ${\bf k}^\to_0 \cgeq {\bf k}, {\bf s}^\to_0 \cgeq {\bf s}$ and ${\bf a}^\to_0 \cgeq {\bf a}$.

    For $\mathbf{k}^\to\in S^\to$ note that \begin{eqnarray*}
        \mathbf{k}^\to_0&= & \bigcurlywedge_{a,b} (a_0\to(b_0\to a_0\curlywedge b_1\to a_1))\curlywedge a_1\to b_1\to a_1\\
        &\succcurlyeq & \bigcurlywedge_{a,b} (a_0\to(b_0\to a_0\curlywedge b_1\to a_0))\curlywedge a_1\to b_1\to a_1\\
        &= & \bigcurlywedge_{a,b} a_0\to b_1\to a_0\curlywedge a_1\to b_1\to a_1\\
        &\succcurlyeq & \mathbf{k}\in S
    \end{eqnarray*}

    Writing 
    \begin{eqnarray*}
        \sigma_0 & = & \bigcurlywedge_{a,b,c} \big( \, (a_0\to (b_0\to c_0\curlywedge b_1\to c_1))\curlywedge a_1\to b_1\to c_1 \, \big)\to \\
        & & \big( \, (a_0\to b_0\curlywedge a_1\to b_1)\to (a_0\to c_0) \, \big) \\
        \sigma_1 & = & \bigcurlywedge_{a,b,c} \big( \, (a_0\to (b_0\to c_0\curlywedge b_1\to c_1))\curlywedge a_1\to b_1\to c_1 \, \big)\to \\
        & & \big( \, (a_0\to b_0\curlywedge a_1\to b_1)\to (a_1\to c_1) \, \big) \\
        \sigma_2 & = & \bcw_{a,b,c} \big( \, (a_0\to (b_0\to c_0\curlywedge b_1\to c_1))\curlywedge a_1\to b_1\to c_1 \, \big)\to \\
        & & \big( \, (a_1\to b_1)\to a_1\to c_1 \, \big)\\
        \sigma_3 & = & \bcw_{a,b,c} (a_1\to b_1\to c_1)\to (a_1\to b_1)\to a_1\to c_1
    \end{eqnarray*}
    we can calculate that
    \begin{eqnarray*}
       \mathbf{s}^\to_0 & =&\bigcurlywedge_{a,b,c} \big( \, (a_0\to (b_0\to c_0\curlywedge b_1\to c_1))\curlywedge a_1\to b_1\to c_1 \, \big)\to\\
       & & \big( \, ((a_0\to b_0\curlywedge a_1\to b_1)\to (a_0\to c_0\curlywedge a_1\to c_1))\curlywedge (a_1\to b_1)\to a_1\to c_1 \, \big)\\
       & & \curlywedge \big( \, (a_1\to b_1\to c_1)\to (a_1\to b_1)\to a_1\to c_1 \, \big) \\
       & = & \bigcurlywedge_{a,b,c} \big( \, (a_0\to (b_0\to c_0\curlywedge b_1\to c_1))\curlywedge a_1\to b_1\to c_1 \, \big)\to\\
       & & \big( \, (a_0\to b_0\curlywedge a_1\to b_1)\to (a_0\to c_0\curlywedge a_1\to c_1) \, \big)\curlywedge\\
       &   &  \big( \, (a_0\to (b_0\to c_0\curlywedge b_1\to c_1))\curlywedge a_1\to b_1\to c_1 \, \big)\to\\
       &  & \big( \, (a_1\to b_1)\to a_1\to c_1 \, \big) \curlywedge \sigma_3 \\
       & = & \bigcurlywedge_{a,b,c} \big( \, (a_0\to (b_0\to c_0\curlywedge b_1\to c_1))\curlywedge a_1\to b_1\to c_1 \, \big)\to\\
       & & \big( \, (a_0\to b_0\curlywedge a_1\to b_1)\to a_0\to c_0 \, \big)\curlywedge\\
       & & \big( \, (a_0\to (b_0\to c_0\curlywedge b_1\to c_1))\curlywedge a_1\to b_1\to c_1 \, \big)\to\\
       & & \big( \, (a_0\to b_0\curlywedge a_1\to b_1)\to  a_1\to c_1 \, \big)\curlywedge \sigma_2 \curlywedge \sigma_3 \\
       & = & \sigma_0 \curlywedge \sigma_1 \curlywedge \sigma_2 \curlywedge \sigma_3
        \end{eqnarray*}
    using that $A$ is binary implicative. So we need to show that $\sigma_i \cgeq {\bf s}$ for all $i \in \{ 0, 1, 2, 3 \}$, which is clear for $i = 3$. For $i = 2$ note that 
    \begin{eqnarray*}
       & & ((a_0\to (b_0\to c_0\curlywedge b_1\to c_1))\curlywedge a_1\to b_1\to c_1)\to ((a_1\to b_1)\to a_1\to c_1)\\
       & \succcurlyeq & (a_1\to b_1\to c_1)\to (a_1\to b_1)\to a_1\to c_1,
    \end{eqnarray*}
    while for $i \in \{ 0, 1 \}$ we have
    \begin{eqnarray*}
         &  & ((a_0\to (b_0\to c_0\curlywedge b_1\to c_1))\curlywedge a_1\to b_1\to c_1)\to ((a_0\to b_0\curlywedge a_1\to b_1)\to (a_i\to c_i)) \\
         & \cgeq & (a_i \to b_i \to c_i) \to (a_i \to b_i) \to (a_i \to c_i).
    \end{eqnarray*}

    To prove  $\mathbf{a}^\to_0 \cgeq {\bf a}$, note that 
    \begin{eqnarray*}
    \mathbf{a}^\to_0 & = & \big(\bigcurlywedge_{x,I, y_i,z_i}\big(\bigcurlywedge_{i}x\to y_i\to z_i\big)_0\to \big(x\to\bigcurlywedge_{i}y_i\to z_i\big)_0\big)\curlywedge\mathbf{a}^\to_1
    \end{eqnarray*}
    As $\mathbf{a}^\to_1 \succcurlyeq \mathbf{a} \in S$, it suffices to show $$\mathbf{a}\preccurlyeq\big(\bigcurlywedge_{x,y_i,z_i}\big(\bigcurlywedge_{i}x\to y_i\to z_i\big)_0\to \big(x\to\bigcurlywedge_{i}y_i\to z_i\big)_0\big)$$

So we calculate \begin{displaymath} \begin{array}{cl}
    &\big(\bigcurlywedge_{i}x\to y_i\to z_i\big)_0\to \big(x\to\bigcurlywedge_{i}y_i\to z_i\big)_0\\ \\
    =&\Big(\big(\bigcurlywedge_i x_0\to(y_{i0}\to z_{i0}\curlywedge y_{i1}\to z_{i1})\big)\curlywedge\big(\bigcurlywedge_i x_1\to y_{i1}\to z_{i1}\big)\Big)\to\\ \\
    &\Big(\big(x_0\to\bigcurlywedge_i y_{i0}\to z_{i0}\curlywedge y_{i1}\to z_{i1}\big)\curlywedge\big(x_1\to\bigcurlywedge_i y_{i1}\to z_{i1}\big)\Big)\\ \\
    =&\Big(\big(\big(\bigcurlywedge_i x_0\to(y_{i0}\to z_{i0}\curlywedge y_{i1}\to z_{i1})\big)\curlywedge\big(\bigcurlywedge_i x_1\to y_{i1}\to z_{i1}\big)\big)\to\\
    &\big(x_0\to\bigcurlywedge_i y_{i0}\to z_{i0}\curlywedge y_{i1}\to z_{i1}\big)\Big)\curlywedge\\
    &\Big(\big(\big(\bigcurlywedge_i x_0\to(y_{i0}\to z_{i0}\curlywedge y_{i1}\to z_{i1})\big)\curlywedge\big(\bigcurlywedge_i x_1\to y_{i1}\to z_{i1}\big)\big)\to\\
    &\big(x_1\to\bigcurlywedge_i y_{i1}\to z_{i1}\big)\big)\Big)
\end{array} \end{displaymath}
Here the final equality follows from binary implicativity. The second part we can be bound from below by $$\mathbf{a}\preccurlyeq\big(\bigcurlywedge_i x_1\to y_{i1}\to z_{i1}\big)\to
    x_1\to\bigcurlywedge_i y_{i1}\to z_{i1}$$
So we only have to bound the first part from below by $\mathbf{a}$.
\begin{displaymath}
\begin{array}{cl}
    &\big(\big(\bigcurlywedge_i x_0\to(y_{i0}\to z_{i0}\curlywedge y_{i1}\to z_{i1})\big)\curlywedge\big(\bigcurlywedge_i x_1\to y_{i1}\to z_{i1}\big)\big)\to\\ \\
    &\big(x_0\to\bigcurlywedge_i y_{i0}\to z_{i0}\curlywedge y_{i1}\to z_{i1}\big)\\ \\
    \succcurlyeq&\big(\bigcurlywedge_i x_0\to(y_{i0}\to z_{i0}\curlywedge y_{i1}\to z_{i1})\big)\to\big(x_0\to\bigcurlywedge_i y_{i0}\to z_{i0}\curlywedge y_{i1}\to z_{i1}\big)\\ \\
    =&\big(\bigcurlywedge_i x_0\to(y_{i0}\to z_{i0}\curlywedge y_{i1}\to z_{i1})\big)\to\big(x_0\to\bigcurlywedge_i y_{i0}\to z_{i0}\big)\curlywedge\\ \\
    &\big(\bigcurlywedge_i x_0\to(y_{i0}\to z_{i0}\curlywedge y_{i1}\to z_{i1})\big)\to\big(x_0\to\bigcurlywedge_i y_{i1}\to z_{i1}\big)\\ \\
    =&\Big(\big(\big(\bigcurlywedge_i x_0\to y_{i0}\to z_{i0}\big)\curlywedge\big(\bigcurlywedge_i x_0\to y_{i1}\to z_{i1}\big)\big)\to\big(x_0\to\bigcurlywedge_i y_{i0}\to z_{i0}\big)\Big)\curlywedge\\ \\
    &\Big(\big(\big(\bigcurlywedge_i x_0\to y_{i0}\to z_{i0}\big)\curlywedge\big(\bigcurlywedge_i x_0\to y_{i1}\to z_{i1}\big)\big)\to\big(x_0\to\bigcurlywedge_i y_{i1}\to z_{i1}\big)\Big)\\ \\
    \succcurlyeq&\Big(\big(\bigcurlywedge_i x_0\to y_{i0}\to z_{i0}\big)\to x_0\to\bigcurlywedge_i y_{i0}\to z_{i0}\Big)\curlywedge\\ \\
    &\Big(\big(\bigcurlywedge_i x_0\to y_{i1}\to z_{i1}\big)\to x_0\to\bigcurlywedge_i y_{i1}\to z_{i1}\Big)\\
    \succcurlyeq & \mathbf{a}
\end{array}
\end{displaymath}
This completes the proof.
    \end{proof}
\end{prop}

\end{document}